\documentclass[10pt,journal,compsoc]{IEEEtran}
\usepackage{hyperref}

\usepackage{xcolor}
\usepackage[american]{babel}
\usepackage{epsfig}
\usepackage{color}
\usepackage{amsthm,amsfonts,amsmath,amssymb,amstext,latexsym}
\usepackage{color}

\allowdisplaybreaks

\newtheorem{theorem}{Theorem}
\newtheorem{lemma}{Lemma}

\newtheorem{remark}{Remark}

\newtheorem{proposition}{Proposition}
\newtheorem{definition}{Definition}

\newtheorem{assumption}{Assumption}

\newcommand{\bydef}{:=}
\newcommand{\E}{\mathbb{E}}

\newcommand{\SO}{\mathcal{S}_{{\rm subopt}}}

\newcommand{\pp}{\mathbb{P}}
\newcommand{\I}[1]{\mathbb{I}_{\{#1\}}}
\newcommand{\II}[2]{\mathbf{1}_{#1}^{#2}}

\ifCLASSOPTIONcompsoc
  \usepackage[nocompress]{cite}
\else
  \usepackage{cite}
\fi
\ifCLASSINFOpdf
\else
\fi
\newif\ifSOC
\SOCfalse

\ifSOC
    \newcommand{\review}[1]{{\color{blue}#1}}
\else
    \newcommand{\review}[1]{{\color{black}#1}}
\fi

\begin{document}
%
\title{Asynchronous Load Balancing and Auto-scaling: Mean-Field Limit and Optimal Design}
%
%
%
%

\author{Jonatha~Anselmi
\IEEEcompsocitemizethanks
{
\IEEEcompsocthanksitem J. Anselmi is with
Univ. Grenoble Alpes, CNRS, Inria, Grenoble INP, LIG, 38000 Grenoble, France.
E-mail: jonatha.anselmi@inria.fr
}
}

\IEEEtitleabstractindextext{%
\begin{abstract}
We develop a Markovian framework for load balancing that combines classical algorithms such as Power-of-$d$ with auto-scaling mechanisms that allow the net service capacity to scale up or down in response to the current load on the same timescale as job dynamics. Our framework is inspired by serverless platforms, such as Knative, where servers are software functions that can be flexibly instantiated in milliseconds according to scaling rules defined by the users of the serverless platform. The main question is how to design such scaling rules to minimize user-perceived delay performance while ensuring low energy consumption. For the first time, we investigate this problem when the auto-scaling and load balancing processes operate \emph{asynchronously} (or \emph{proactively}), as in Knative. In contrast to the synchronous (or reactive) paradigm, asynchronism brings the advantage that jobs do not necessarily need to wait any time a scale-up decision is taken.

In our main result, we find a general condition on the structure of scaling rules able to drive mean-field dynamics to delay and relative energy optimality, i.e., a situation where both the user-perceived delay and the relative energy waste induced by idle servers vanish in the limit where the network demand grows to infinity in proportion to the nominal service capacity. The identified condition suggests to scale up the current net capacity if and only if the mean demand exceeds the rate at which servers become idle and active.
Finally, we propose a family of scaling rules that satisfy our optimality condition.  Numerical simulations demonstrate that these rules provide better delay performance than existing synchronous auto-scaling schemes while inducing almost the same power consumption.
\end{abstract}

\begin{IEEEkeywords}
Load balancing, auto-scaling, serverless computing, asymptotic optimality, Knative
\end{IEEEkeywords}}

\maketitle

\IEEEdisplaynontitleabstractindextext

%
\IEEEpeerreviewmaketitle

\IEEEraisesectionheading{\section{Introduction}\label{sec:introduction}}

%
%
%
%



\IEEEPARstart{L}{oad} balancing
is the process of distributing work units (jobs) over a set of distributed computational resources (servers) for processing.
In large architectures, each server has its own queue, as this enhances scalability, and jobs are irrevocably dispatched to one out of~$N$ parallel servers instantaneously upon their arrival.
%
%
Given the stringent latency requirements of modern applications, breaches of which can severely impact revenue,
load balancing techniques are designed to optimize user-perceived delay performance and popular examples are Power-of-$d$ \cite{Mitzenmacher2001} and Join-the-Idle-Queue (JIQ) \cite{Lu2011}.

Closely related to load balancing, \emph{auto-scaling} is a term often used in cloud computing
to refer to the process of adjusting
the current service capacity
automatically in response to the current load~\cite{Buyya2018}.
Auto-scaling mechanisms are meant to control the current net capacity over time to avoid
performance degradation, which yields unacceptably large delays, and
overprovisioning of resources, which yields high infrastructure and energy costs.
%
%
Google Cloud Run,
Amazon Elastic Compute Cloud (EC2), Microsoft Windows Azure and Oracle Cloud Platform are examples of platforms that offer auto-scaling and load balancing features.
Users of these platforms deploy their applications with some control on how the system should scale up resources in front of an increased load. Modern auto-scaling mechanisms are extremely reactive in the sense that they control the current net capacity relying on fresh observations of the system state rather than historical data.
This especially holds true in \emph{serverless computing platforms}, or Function-as-a-Service, which nowadays provide the convenient solution to deploy any type of application or backend service~\cite{scaleperrequest}.

In this paper, we are interested in the interplay between the load balancing and auto-scaling processes.
The main objective is to design a scheme that combines both to minimize delay performance while ensuring low energy consumption.

\subsection{Timescale Separation}

Most of the existing performance models for load balancing assume that the available service capacity remains constant over time~\cite{van2018scalable}, i.e.,
auto-scaling is not taken into account.
Nonetheless, auto-scaling mechanisms are widely employed by cloud applications and affect delay performance.
This does not mean that classic load balancing models are inadequate for cloud systems but simply that they assume that auto-scaling operates at a much slower \emph{timescale} than load balancing.
Essentially, this means that jobs do not see any change in the available capacity because they evolve much faster than servers.
This makes sense if servers are interpreted as physical or even virtual machines because setup times are of the order of minutes if not longer~\cite{Gandhi13}
%
while in typical applications hosted in cloud networks job service times are about ten milliseconds~\cite{scaleperrequest}.
The large body of literature on load balancing, reviewed in Section~\ref{sec:RW}, is undoubtedly the proof that this {timescale separation} assumption is well accepted for several systems.
In the context of serverless computing however,
a server is interpreted as a software function that can be flexibly instantiated in milliseconds \cite{Peeking18,Peeking18}, i.e., within a time window that is comparable with the magnitude of job inter-arrival and service times, and with negligible switching costs.
Here,
auto-scaling mechanisms are extremely reactive and the decisions of turning servers on or off are based on instantaneous observations of the current system state rather than on the long-run equilibrium behavior.
Therefore,
the timescale separation assumption above becomes questionable, as also discussed in~\cite{scaleperrequest},
because
it would mean to assume that job dynamics achieve stochastic equilibrium between consecutive changes of the net service capacity, i.e., in milliseconds.
%
%
%
%
%
%

\subsection{Getting Rid of the Timescale Separation Assumption}

While a large body of the literature investigates load balancing and auto-scaling \emph{separately} \cite{van2018scalable,Buyya2018}, little has been done when both are applied jointly within the same timescale.
Existing works focus on \emph{synchronous} (i.e., both scale-up and dispatching decisions are taken at the same time) or \emph{centralized} (i.e., all servers share a common queue) architectures~\cite{Gandhi13,scaleperrequest}.
%
%
For scalability reasons however,
no central queue is maintained {(in this case, we say that the architecture is \emph{decentralized})}
and no decisions are taken synchronously
in massive cloud systems;
see Section~\ref{sec:sync_vs_async} below.
A decentralized but synchronous architecture where JIQ is synchronized with an ad-hoc auto-scaling strategy is considered in \cite{elasticSIG,Jonckheere18,Clausen}.
%
%
In contrast,
we consider a decentralized and \emph{asynchronous} architecture, where the term ``asynchronous'
means that scaling and dispatching decisions are decoupled.



\subsection{Synchronous vs Asynchronous in Serverless Computing}
\label{sec:sync_vs_async}


The load balancing and auto-scaling processes of existing implementations of public serverless computing platforms are either ``synchronous'' or ``asynchronous''; this terminology is borrowed from the cloud computing community~\cite{scaleperrequest}, though some works use the terms ``reactive'' and ``proactive'', respectively~\cite{DOGANI2024104837}.
As explained in these references,

\begin{itemize}

\item

The auto-scaling principle underlying a synchronous architecture is that
\emph{a new server is turned on at the arrival time of a job if the job itself finds all servers busy}.
The drawback of this approach is that all jobs that have triggered a scale-up signal are forced to wait before being processed.
In centralized implementations,
each of these jobs waits for the activation of the server that has been launched at the moment of its arrival (coldstart latency) \cite{254430,Peeking18,scaleperrequest},
while in the decentralized proposals given in \cite{elasticSIG,Jonckheere18,Clausen}, each of these is sent to an already active (busy) server chosen at random,
hence slowed down by the jobs ahead.

To the best of our knowledge, no synchronous-decentralized implementations currently exist.
In contrast, AWS Lambda, Azure Functions, IBM Cloud Functions and Apache OpenWhisk are examples of synchronous-centralized platforms.

\item

The auto-scaling principle underlying an asynchronous architecture is that
\emph{the load balancing and auto-scaling processes are decoupled}.
Specifically, a job is dispatched to some running server immediately upon its arrival  according to some load balancing algorithm and,
independently of this, an auto-scaling mechanism decides whether the current processing capacity should change as a function of user-defined metrics that may depend on instantaneous observations of the current system state \cite{scaleperrequest}.
Because of this decoupling, scale-up decisions do not need to wait that all active servers are busy as in the synchronous approach. Thus, they may anticipate the arrival of a job and overcome the intrinsic drawback of the synchronous approach described above.
In addition, the scale-up decision rate is fine-tuned by the platform user; in Knative, this is set via the \texttt{max-scale-up-rate} global key.

To the best of our knowledge, no asynchronous-centralized implementations currently exist.
In contrast, Google Cloud Run and Knative are examples of asynchronous-decentralized platforms~\cite{KnativeDocs}.

\end{itemize}

In a stochastic and dynamic setting, no performance model/analysis is available in the literature for the asynchronous-decentralized approach. Our  main motivation is to contribute to fill this gap.

\subsection{Summary of our Contributions}

We develop a Markovian framework for load balancing that includes asynchronous auto-scaling mechanisms.
We refer to this framework as `Asynchronous Load Balancing and Auto-scaling' (ALBA).
%
%
Two (asynchronous) mechanisms drive dynamics in ALBA:
\begin{itemize}


\item[i)]  a \emph{dispatching rule}, or load balancing rule, which defines how jobs are dispatched among the set of active servers as they join the system, and

 \item[ii)] a \emph{scaling rule}, which defines how the number of active servers scales up and down over time, possibly as a function of the current system state.

 \end{itemize}
%
%
%
The dispatching rules included in ALBA are Join-Below-Threshold-$d$ (JBT-$d$), which is a generalization of JIQ, and Power-of-$d$; in fact, these are the rules used in Knative~\cite{KnativeLB}.
We also assume that a server is turned off only if it remains idle during an expiration window.
This scale-down rule is commonly used in practice~\cite{scaleperrequest,KnativeDocs} and also known as ``delay-off''~\cite{Gandhi13}.
%
In contrast, we do not impose any particular structure on scale-up rules because they are usually defined by the user of the serverless platform.
%
%
%
Having fixed the scale-down rule, in the following the term ``scaling rule'' refers to a scale-up rule.

%

Our key technical contribution is a general condition on the structure of scaling rules that is able to drive the mean-field dynamics induced by ALBA to \emph{delay and relative energy optimality}, a situation where the user-perceived delay  and the relative energy wastage induced by idle servers vanish.
This condition suggests to scale up capacity if and only if the mean demand exceeds the overall rate at which servers become idle and active, which can be measured.

We also propose \emph{Rate-Idle}, see Definition~\ref{blind}, a scaling rule that satisfies our optimality condition.
{Provided that it is combined with JIQ, we show by means of numerical simulations that Rate-Idle
provides a better delay performance than the synchronous schemes in~\cite{elasticSIG,Jonckheere18} while inducing the same energy consumption cost.
We own this gain to the fact that scale up decisions may be taken before job arrivals, while in a synchronous scheme such as TABS-$d$, jobs are forced to wait any time a scale up decision is taken.
}

Our results are obtained through a rigorous analysis of the underlying Markov process in the mean-field limit.
Here, we establish the convergence of the stochastic finite model to a fluid model with a discontinuous drift.
Then, we leverage the fluid model to identify a condition that drives the fluid trajectories to a unique fixed point corresponding to delay and relative energy optimality.

\subsection{Organization and Main Results Detailed}


Section~\ref{sec:RW} reviews the existing literature and Section~\ref{sec:ALBA} introduces ALBA by defining a stochastic (intractable) and a deterministic (tractable) model to describe its dynamics.
Section~\ref{sec:main_results} presents our main results, i.e., Theorems~\ref{th1}, \ref{th2} and~\ref{th3}:


\begin{itemize}


\item Theorem~\ref{th1} connects the stochastic and the deterministic models
and justifies the use of the latter to approximate the dynamics of the former.
This enables analytical tractability and allows one to study dynamics easily.
%
We prove Theorem~\ref{th1} following the framework developed in~\cite{tsitsiklis2012,Bramson1998}, though we develop ad-hoc arguments to handle the discontinuities of the drift function of the underlying Markov chain.

 \item
Theorem~\ref{th2} characterizes the \emph{fixed points} of the deterministic model in terms of a set of non-linear equations. It also provides a simple necessary and sufficient condition able to tell whether or not the nominal service capacity will be needed to handle the incoming demand.
Within Power-of-$d$, roughly speaking, there always exists a unique fixed point if the scaling rule is ``nice''.
Within JBT-$d$ however, uniqueness is guaranteed only if the scaling rule has access to the number of servers containing exactly one job (see Remark~\ref{sc8as9s}).



\item
Theorem~\ref{th3} investigates how to design \emph{optimal} and {globally stable} scaling rules. More specifically, we identify a general condition ensuring that dynamics of the deterministic model converge to delay and relative energy optimality.
We show that optimality can only be achieved within JIQ (or equivalently JBT-$0$), though in practice this may not be the convenient choice within architectures with several dispatchers. In this case, an exact implementation of JIQ would imply an expensive communication overhead per job and Power-of-$d$ may be the way to go as it does not require the dispatcher(s) to store information about the server states.

\end{itemize}

%
%





Section~\ref{sec:empirical} compares by simulation the asynchronous and synchronous approaches, showing that the former provides a much better delay performance.
Then, Section~\ref{sec:numerical}
develops a tractable optimization framework to illustrate how the results presented in this paper can be applied to trade off between performance and energy consumption.
Finally, Section~\ref{sec:conclusions} draws the conclusions.
Proofs of our results are deferred to the appendix.

\section{Literature review}
\label{sec:RW}


The existing literature related to load balancing and auto-scaling is huge and our goal is to provide the necessary background highlighting the difference of our work.

\subsection{Load Balancing and the Zero Delay Property}

Popular examples of load balancing algorithms that work well when servers are homogeneous, i.e., all servers have the same processing speed, are Random, Round-Robin (RR) \cite{LR98,anselmi:hal-01614892}, Power-of-$d$ \cite{Mitzenmacher2001}, Join-the-Idle-Queue (JIQ) \cite{Lu2011}, Least-Left-Workload (LLW) and Size Interval Task Allocation (SITA) \cite{anselmi2019b,HarcholBalter2009,SITAE}.
Random sends each job to random server,
RR sends jobs to servers in a cyclic manner,
Power-of-$d$ sends an incoming job to the least loaded server among $d$ selected uniformly at random.
JIQ sends an incoming job to a random idle server if an idle server exists and to a random one otherwise,
LLW sends an incoming jobs to the queue having the shortest workload, and
SITA sends a job to a given server if its size belong to a given interval.
In general, it is not possible to identify which of these algorithms is the best because the general answer depends on the underlying architecture, load conditions, service time distribution and on the amount of information available to the dispatcher~\cite{van2018scalable}.
%
%


Recently, a number of works attempted to understand under which conditions the mean waiting time can be driven down to zero in the limiting regime where the arrival rate grows linearly with the number of servers while keeping the average load below one.
This is possible within different load balancing schemes and architectures.
Examples include JIQ \cite{Stolyar2015}, Power-of-$d$ with $d\to\infty$ as the network size grows to infinity \cite{Borst2016}, Power-of-$d$ with memory \cite{anselmi2019}, SITA combined with RR~\cite{Anselmi2020} and the pull-based policies developed in~\cite{Gamarnik2016,borst_hyper}.
To some extent, the fundamental limits of load balancing are described in~\cite{Gamarnik2016}, where the authors investigate trade-offs between performance (the zero-delay property), communication overhead and memory within a certain class of symmetric architectures and the large-system limiting regime.

\subsection{Joint Load Balancing and Auto-scaling}

The load balancing algorithms above have been analyzed under the assumption that the active number of servers is \emph{constant} at all times.
Few works considered a time-varying net capacity~\cite{elasticSIG,Jonckheere18,Clausen,DebankurStolyar}.
In these references,
JIQ is synchronized with a specific auto-scaling strategy as described in Section~\ref{sec:sync_vs_async}.
When the traffic demand and the nominal service capacity proportionally grow to infinity,
the mechanism proposed in \cite{elasticSIG} yields the zero-delay property but also deactivates any surplus idle servers,
thus inducing delay and relative energy optimality.
This property has been strengthened in~\cite{DebankurStolyar}, where the authors relax some finite buffer assumptions.
In contrast, our work shows optimality:
\begin{itemize}
 \item
within an asynchronous (see Section~\ref{sec:sync_vs_async}) architecture; an advantage of asynchronism is that jobs do not necessarily need to wait any time a scale-up decision is taken, a fact whose performance gain is evaluated in Section~\ref{sec:empirical} by simulation;

\item without limiting on an ad-hoc auto-scaling strategy; rather, we identify a structural property on scaling rules that induces optimality under broader conditions (Theorem~\ref{th3}).
\end{itemize}

\section{Asynchronous Load Balancing and Auto-scaling (ALBA)}
\label{sec:ALBA}


In this section, we first describe the main principles at the basis of Asynchronous Load Balancing and Auto-scaling (ALBA).
Since our aim is to develop a model tailored to serverless computing,
we will make several references to Knative, a popular serverless framework for hosting Function-as-a-Service processing that is used, among others, by Google Cloud Run.
%
Then,
we propose two performance models for ALBA.
The first is meant to capture the stochastic nature of the underlying dynamics
while the second is deterministic and will serve to approximate the dynamics induced by the first.
The advantage of the deterministic model is its tractability.
Finally, we formalize the structure of the scaling rules investigated in this paper.


\subsection{System Description}
\label{sec:sys_desc}

The proposed framework, ALBA, is composed of a system of $N$ parallel {servers}, each with its own queue, that represent the nominal service capacity, i.e., the upper limit on the amount of resources that one user can have up and running at the same time\footnote{In Knative, this upper limit is specified by the \texttt{max-scale-limit} global key.}.
In the cloud computing community, servers are also referred to as containers, cloud functions, instances or replicas.
Public serverless computing platforms usually require to specify such limit in order to ensure service availability for other users.
In the following, the terms servers and queues will be used interchangeably.
A server is said \emph{warm} if turned on, \emph{cold} if turned off and \emph{initializing}
if making the transition from cold to warm.
These are the possible server states~\cite{Peeking18,Optimizing19,scaleperrequest}.
An initializing server performs basic startup operations such as connecting to database, loading libraries, etc. This is the time to provision a new function instance.
Only warm servers are allowed to receive jobs.
A server is also said \emph{idle-on} if warm but not processing any job, and \emph{busy} if warm and processing some job.
Typically, billing policies charge per number of warm and initializing servers used per time unit.

Jobs join the system from an exogenous source to receive service. Upon arrival, each job is dispatched to a warm server according to some dispatching rule.
After dispatching, each job is processed by the selected server according to the presumed scheduling discipline at that server.
After processing, each job leaves the system.

\begin{assumption}
\label{as_dispatcher}
Jobs are dispatched to servers according to either Power-of-$d$ or Join-Below-Threshold-$d$ (JBT-$d$).
\end{assumption}
We recall that
Power-of-$d$ sends an incoming job to the shortest among $d\ge 1$ warm servers selected at random at the moment of its arrival and
JBT-$d$ sends an incoming job to a warm server containing no more than $d\ge 0$ jobs if one exists otherwise to a warm server selected at random.
In all cases, ties are broken randomly.
If $d=0$, JBT-$d$ is also known as Join-the-Idle-Queue (JIQ) \cite{Lu2011}.
We limit our framework to these types of schemes
because they involve a constant communication overhead per job (in architectures with a single dispatcher) and
because they are commonly used in practice.
For instance, Knative uses Power-of-$2$ if no limit is set on the queue length of each server
and JBT-$d$ if such limit is set to $d$~\cite{KnativeLB}.

Alongside with the above job dynamics, the pools of warm/initializing/cold servers change over time in the background and in an asynchronous manner. Precisely, the platform monitors the system state at some epochs that we refer to as \emph{{scaling} times}.
At such times, a cold server is selected, provided that one exists, and becomes initializing
according to the outcome of some scaling rule.
After some \emph{initialization time}, or coldstart latency, an initializing server becomes idle-on.
When a server becomes idle-on, it becomes cold after a scale down delay, or \emph{expiration time}, if during such time the server received no job; this scale-down rule is used in several serverless computing platforms (including Knative) \cite{Peeking18,SPEC18} and also in other settings \cite{Gandhi13}.
%
%
%
We observe that the number of warm servers fluctuates from 0 to $N$ over time.
While in practice it may be possible to set a lower limit on the number of warm servers,
the scale down to zero (or one) servers configuration is usually the default choice
\cite{Knative-scale-bounds}.


To a great extent, the scale up rule, the expiration rate and the {scaling} times are under the control of the platform user, which may design them in a way to optimize a trade-off between performance and energy.
On the other hand, several measurements indicate that initialization times are typically one order of magnitude higher than jobs' service times in serverless platforms~\cite{scaleperrequest,Peeking18}.


\subsection{Notation}
\label{sec:notation}

We introduce some notation that will be used throughout the paper.
Let $B\in \mathbb{Z}_+\cup\{+\infty\}$ be a constant that will denote the buffer size of each server.
We use $\I{A}$ to denote the indicator function of~$A$.
If $a\in\mathbb{R}$ and $A$ denotes an interval, $\mathbf{1}_{A}^a\bydef \I{a\in A}$.
We also let $(\cdot)^+\bydef \max\{\cdot,0\}$ and $\|\cdot\|$ denotes the $L_1$ norm.
Unless specified otherwise,
 $(i,j)$ ranges over the set $\{0,\ldots,B\}\times\{0,1,2\}$ if $B<\infty$ and over $\mathbb{Z}_+\times\{0,1,2\}$ otherwise.
%
%
%
%
The process of interest will take values in
$\mathcal{S}\bydef
\{ (x_{i,j}\in\mathbb{R}_+, \forall (i,j)) : \sum_{i,j} x_{i,j} =1
\}$
and our analysis holds under the distance function $d_w$ induced by the weighted $\ell_2$ norm $\|\cdot\|_w$ on $\mathbb{R}^{\mathbb{Z}_+}$ defined by
$\|x-x'\|_w^2 \bydef \sum_{i,j} \frac{|x_{i,j}-x_{i,j}'|^2}{2^{i+j}}.$
For $x\in\mathcal{S}$, let $y_i\bydef \sum_{k\ge i} x_{i,2}$.
We also let
$\mathcal{S}_1\bydef \{x\in\mathcal{S}: \sum_{i\ge 1} i x_{i,2}<\infty\}$.

\subsection{Markov Model}
\label{sec:model}

We model the dynamics induced by ALBA in terms of a continuous time Markov chain.
{The exogenous arrival process of jobs is assumed to be Poisson with rate~$\lambda N$, with $0<\lambda<1$.
Our analysis (Theorem~\ref{th1}) generalizes trivially to a time-varying arrival rate, a case that we omit
for clarity of exposition.
We discuss this point in the Conclusions.}
%
%
%
%
%
%
The processing times, or service times, of jobs are independent and exponentially distributed random variables with unit mean. Servers process jobs according to any work-conserving discipline.
Upon arrival, each job is assigned to one warm server as specified in Assumption~\ref{as_dispatcher}.
In the extreme case where no warm server exists, the job is lost.
We assume that each server can contain at most $B>d$ jobs and a job that is sent to a server with $B$ jobs is rejected.
If not specified otherwise,~$B$ is either finite or infinite.
At each {scaling} time, a cold server is selected uniformly at random,
provided that one exists, and becomes initializing with some probability~$g$. This is the \emph{scaling probability} (or rule) and will possibly depend on the system state; in the conclusion section, we will discuss how our work adapts to the case where a random number of cold servers is selected at each scaling time.
Given that jobs arrive with a rate proportional to~$N$ and only one server can be added at each scaling time, we let the scaling frequency increase with~$N$ as well.
As it occurs in Knative, this implies that the number of servers created in a time window of constant size is proportional to~$N$ if within such window the scaling probability is not zero.
We let the inter-{scaling}, initialization and expiration times be independent and exponentially distributed with rate $\alpha N$, $\beta$ and $\gamma$, respectively.


Let $\tilde{Q}^N(t)\bydef (\tilde{Q}_1^N(t),\ldots,\tilde{Q}_N^N(t))$ be the vector of queue lengths at time $t$, including the jobs in service, and let
$\tilde{S}^N(t)\bydef (\tilde{S}_1^N(t),\ldots,\tilde{S}_N^N(t))$
be the vector of server states.
Specifically, $\tilde{S}_k^N(t)\in\{0,1,2\}$ indicates whether server $k$ is cold ($\tilde{S}_k^N(t)=0$), initializing ($\tilde{S}_k^N(t)=1$) or warm ($\tilde{S}_k^N(t)=2$) at time $t$.
Under the above assumptions, the stochastic process $(\tilde{Q}^N(t),\tilde{S}^N(t))$ is a continuous-time Markov chain on state space
$
\{(n,s)\in\{0,\ldots,B\}^N\times \{0,1,2\}^N: n_k>0 \Rightarrow s_k=2,\, \forall k=1,\ldots,N\}$.

It is
convenient to describe dynamics in terms of the process
$X^N(t)\bydef(X_{0,0}^N(t),X_{0,1}^N(t), X_{i,2}^N(t):  i =0,\ldots,B)$
where
\begin{align}
X_{i,j}^N(t) \bydef \frac{1}{N}\sum_{k=1}^N \I{ \tilde{Q}_k^N(t)=i, \tilde{S}_k^N(t)=j}
\end{align}
is the proportion of servers in state $j$ with $i$ jobs at time $t$.
The process~$X^N(t)$ is still a Markov chain with values in some set $\mathcal{S}^{(N)}$
that is a subset of $\mathcal{S}$.
Let $e_{i,j}\bydef \left(\delta_{i,i'}\,\delta_{j,j'} \in\{0,1\} : i'\ge 0, j'=0,1,2 \right)$  where $\delta_{a,b}$ denotes the Kronecker delta
and
let~$x\bydef (x_{i,j}) \in \mathcal{S}^{(N)}$ denote a generic state of $X^N(t)$.
For conciseness, the Markov chain $X^N(t)$ has the following transitions:
\begin{equation*}
\begin{array}{lll}
x\mapsto x'\bydef x + \frac{1}{N}(e_{i,2}-e_{i-1,2})   &{\rm with~rate}\quad \lambda N f_{i-1}(x) \\
\vspace{-0.13cm}
\\
x\mapsto x'\bydef x + \frac{1}{N} (e_{i-1,2}-e_{i,2} )    &{\rm with~rate}\quad x_i N  \\
\vspace{-0.13cm}
\\
x\mapsto x'\bydef x + \frac{1}{N} (-e_{0,0}, e_{0,1})    &{\rm with~rate}\quad \alpha N g
\\
\vspace{-0.13cm}
\\
x\mapsto x'\bydef x + \frac{1}{N} (-e_{0,1}, e_{0,2})    &{\rm with~rate}\quad \beta x_{0,1} N \\
\vspace{-0.13cm}
\\
x\mapsto x'\bydef x + \frac{1}{N} (e_{0,0} - e_{0,2})  &{\rm with~rate}\quad \gamma x_{0,2} N
\end{array}
\end{equation*}
for all $i=1,\ldots,B$,
provided that $x,x'\in\mathcal{S}^{(N)}$.
Here,
$g\bydef g(x):\mathcal{S}\to[0,1]$ is the scaling probability,
and
$f_{i}(x)$, which depends on the dispatching rule, represents the probability of assigning an incoming job to a warm server containing exactly~$i$ jobs.
If $y_0>0$,
within Power-of-$d$ we have (assuming that server selections are with replacement)
\begin{equation}
\label{eq:h_pod}
%
f_i(x) =\frac{y_i^d-y_{i+1}^d}{y_0^d},
\end{equation}
where  $y_i\bydef y_i(x)\bydef\sum_{j\ge i}x_{j,2}$,
and within JBT-$d$ we have
\begin{equation}
\label{eq:h_jbt}
%
f_i(x) =
\frac{x_{i,2} \, \I{\sum_{k=0}^d x_{k,2}=0}}{y_0}
+
\frac{x_{i,2} \, \I{\sum_{k=0}^d x_{k,2}>0}}{\sum_{k=0}^d x_{k,2}} \I{i\le d},
\end{equation}
where we have taken the convention that $0/0=0$, for all $i=0,\ldots,B-1$.
If $y_0=0$, then $f_i(x)=0$ as no warm server exists.

\subsection{Deterministic Model}

We introduce the deterministic (or fluid, mean-field) model for the dynamics of ALBA.

\begin{definition}
\label{def1}
A continuous function $x(t):\mathbb{R}_+\to \mathcal{S}$ is said to be a \emph{fluid model} (or fluid solution) if for almost all $t\in[0,\infty)$
\begin{subequations}
\label{x_def}
\begin{align}
\label{x00_def}
\dot{{x}}_{0,0} & = \gamma {x}_{0,2} - \alpha g\I{{x}_{0,0}>0}
-  \gamma {x}_{0,2}\,\I{{x}_{0,0}=0,\, \gamma {x}_{0,2} \le \alpha g}\\
\label{x01_def}
\dot{{x}}_{0,1} & = \alpha  g \I{{x}_{0,0}>0} -\beta{x}_{0,1}
+\gamma {x}_{0,2}\,\I{{x}_{0,0}=0,\, \gamma {x}_{0,2} \le \alpha g} \\
\label{x02_def}
\dot{x}_{0,2} &= x_{1,2} - h_0(x) + \beta x_{0,1} - \gamma x_{0,2} \\
\label{xi2_def}
\dot{x}_{i,2} &= x_{i+1,2}\I{i< B}-x_{i,2} + h_{i-1} (x)- h_i(x)\I{i< B},
\end{align}
\end{subequations}
$i=1,\ldots,B$, where
$g\bydef g(x):\mathcal{S}\to[0,1]$,
and
$h_i (x)=\min\{\beta x_{0,1},\lambda\}$ if  $y_0>0$
and otherwise ($y_0=0$):
\begin{align}
\label{h_i_Pod_def}
h_i (x) =
\lambda\, \frac{y_i^d-y_{i+1}^d}{y_0^d}
\end{align}
if Power-of-$d$ is applied and
\begin{align}
\label{h_i_JBTd_def}
h_i (x) =
\left\{
\begin{array}{ll}
%
\lambda\, \frac{x_{i,2}}{\sum_{k=0}^d x_{k,2}} \I{i\le d}, \quad {\rm if }\,\,  \sum_{k=0}^dx_{k,2}>0\\
\\
\left(\beta {x}_{0,1}+{x}_{d+1,2}\I{i=d} \right)\I{{x}_{d+1,2} + (d+1)\beta {x}_{0,1}\le \lambda},
\\ \qquad\qquad\qquad {\rm if }\,\,  \sum_{k=0}^dx_{k,2}=0,\,\, i\le d,\\
\\
 \frac{{x}_{i,2}}{y_0} (\lambda - {x}_{d+1,2} - (d+1) \beta {x}_{0,1}  )^+,
\\ \qquad\qquad\qquad {\rm if }\,\,  \sum_{k=0}^dx_{k,2}=0,\,\, i> d,\\
\end{array}
\right.
\end{align}
if JBT-$d$ is applied.
%
%
%
%
\end{definition}


As for $X_{i,j}^N(t)$, $x_{i,j}(t)$ is interpreted as the proportion of servers in state~$j$ with~$i$ jobs at time~$t$.

%

Let us provide some intuition about the fluid model.
First, when a strictly positive fluid mass of warm server exists, i.e., $y_0>0$, the functions $h_i$ are interpreted as the rate at which jobs are assigned to servers with exactly~$i$ jobs. When the amount of fluid of cold servers is strictly positive, i.e., $x_{0,0}>0$, to some extent these equations may be interpreted as the conditional expected change, or \emph{drift}, from state~$x$ of the Markov chain~$X^N(t)$.
In contrast, when $x_{0,0}=0$, there exists a term, $- \I{{x}_{0,0}=0,\, \gamma {x}_{0,2} \le \alpha g} \gamma {x}_{0,2}$ (see~\eqref{x00_def} and \eqref{x01_def}), that still drains the amount of cold servers down.
This is due to warm servers that become cold but immediately turn initializing
and it appears if the scaling rule is `greedy enough', i.e.,
if the rate at which new initializing servers can be created is greater than or equal to the rate at which warm servers go cold.
This term  is due to fluctuations of order $1/N$ that appear when $X_{0,0}^N(t)=0$, which bring discontinuities in the drift of $X^N(t)$, and will come out from the stochastic analysis developed in Appendix~1.1.3.
%

Now, let us focus on \eqref{h_i_Pod_def} and \eqref{h_i_JBTd_def}, and
let us assume that $y_{0}>0$.
In the case of Power-of-$d$, $h_i=\lambda f_i$ and $x(t)$ evolves following the natural dynamics of Power-of-$d$ as in \cite{Mitzenmacher2001}, though normalized on the variable mass of warm servers~$y_0(t)$.
%
The case of JBT-$d$ is more delicate because of the discontinuous structure of $f_i$ in \eqref{eq:h_jbt}.
If a strictly positive fraction of warm servers with no more than $d$ jobs exist,
then $h_i=\lambda f_i$  and $x(t)$ evolves following the natural dynamics of JBT-$d$, though again normalized on a variable number of servers. On the other hand, when $\sum_{k=0}^dx_{k,2}=0$,
there is a flow of warm servers with at most $d$ jobs that are created but immediately used for dispatching jobs.
Specifically, there are two factors that come into play here:
the first is due to initializing servers that get warm with exactly $i$ jobs (with rate $\beta x_{0,1}$), for all $i\le d$,
and
the second is due service completions from servers with exactly $d+1$ jobs (with rate $x_{d+1,2}$).
The resulting rate
can not be greater than $\lambda$, the rate where jobs are assigned to servers, and this justifies the $\I{{x}_{d+1,2} + (d+1)\beta {x}_{0,1}\le \lambda}$ term.
Then, the excess of such rate, $(\lambda - {x}_{d+1,2} - (d+1)\beta {x}_{0,1})^+$, is distributed uniformly over servers with $i>d$ jobs.
In Theorem~\ref{th3}, we will show that such rate is key for the design of fluid optimal scaling rules.
%
Finally, assume that no warm server exists, i.e., $y_0=0$.
Here,
initializing servers get idle-on with rate $\beta x_{0,1}$ but all of them are immediately filled by new arrivals if $\lambda\ge \beta x_{0,1}$, and in this case the mass of idle-on servers remains zero.
Otherwise, $x_{0,2}$ increases with surplus rate $ \beta x_{0,1} - \lambda$.

The existence of a fluid solution started in $x^{(0)}\in\mathcal{S}_1$ will be direct from Theorem~\ref{th1}.

\subsection{Scaling Rules}
\label{sec:scaling_rules_def}

The scaling rule $g$ gives the probability to activate a new server at each scaling time as a function of the system state. The following assumption,
which will hold throughout the paper,
provides the structure of the scaling rules investigated in this paper.

%
%

\begin{assumption}
\label{as3}
The scaling rule~$g:\mathcal{S}\to [0,1]$ is Lipschitz continuous,
and $g(x)>0$ if $x_{0,0}=1$.
\end{assumption}

The last technical condition
is natural and will rule out
the existence of
degenerate fixed points.
We allow $g(x)$ to be greater than zero even when no cold server exists, i.e., $x_{0,0}=0$.
While this has no impact on the dynamics of the stochastic model, it does affect the fluid model as there may exist a flow of idle-on servers that go cold but instantly turn initializing keeping the proportion of cold servers at zero.
This situation can occur if $\lambda$ is large enough and not only in the transient regime; see Theorem~\ref{th2}.

We propose two scaling rules that satisfy Assumption~\ref{as3}.


\begin{definition}
\label{blind}
At each {scaling} time, if the system state is $x$,
\begin{itemize}
 \item
\emph{Blind}-$\theta$ activates a new server with probability $g(x)=\theta$, $\theta \in(0,1]$;





 \item
\emph{Rate-Idle} activates a new server with probability $g(x)=\frac{1}{\lambda}(\lambda-\beta x_{0,1} -x_{1,2})^+$.

\end{itemize}

\end{definition}

Blind-$\theta$ is oblivious of the system state and thus highly scalable.
Rate-Idle scales resources up if and only if the mean demand, $\lambda$, exceeds the rate at which servers become idle-on, $\beta x_{0,1}+x_{1,2}$.
Here, the auto-scaler  needs to know the amount of initializing servers, the amount of busy servers with exactly one job and both the job arrival and server initialization rates; in Knative, these variables are available to the auto-scaler.
If combined with JIQ, we will show in Theorem~\ref{th3} that Rate-Idle is asymptotically optimal.
%


\section{Main Results}
\label{sec:main_results}

We now present our main results.
In Theorem~\ref{th1}, we justify the use of the deterministic model to approximate the behavior of the stochastic model.
Then, we focus on properties of the deterministic model and i) characterize its fixed points in Theorem~\ref{th2} and ii) investigate the design of optimal scaling rules in Theorem~\ref{th3}.

\subsection{Connection between the Fluid and Markov Models}
\label{sec:connection}

The following result shows that the fluid model can be seen as a first-order approximation of the sample paths of the stochastic \review{model}.
\begin{theorem}
\label{th1}
Let $T<\infty$, $x^{(0)}\in\mathcal{S}_1$ and assume that $\|X^N(0)- x^{(0)}\|_w \to 0$ almost surely.
Then, limit points of the stochastic process $(X^{N}(t))_{t\in[0,T]}$
exist and almost surely satisfy the conditions that define a fluid solution started at $x^{(0)}$.
\end{theorem}
\begin{proof}
Given in Appendix~1.
\end{proof}

The stochastic and the deterministic models have some non-standard aspects that prevent us to prove Theorem~\ref{th1} by directly applying Kurtz's theorem or similar known results. The main technical difficulty is that
the trajectories of the deterministic model may cross or converge to points of discontinuity of its drift function.
%
%
We handle this by following the general framework in~\cite{tsitsiklis2012,Bramson1998} and developing  ad-hoc arguments specific to the structure of our problem (given in Appendix~1.1.3).
%

In view of Theorem~\ref{th1} and since typical and default maximum scale limit values of real applications are 1000 or more \cite{scaleperrequest}, i.e., $N\ge 10^3$, we expect that the fluid model $x(t)$ provides an accurate approximation of the average behavior of~$X^N(t)$.
To support this claim, we present the results of numerical simulations; see also Section~\ref{sec:numerical}.
Figure~\ref{fig:simulations} (left) plots the trajectories of $x(t)$ and $X^{N}(t)$ when $N=10^3$ and $B=10^2$ along the coordinates of cold ($x_{0,0}$), initializing  ($x_{0,1}$), idle-on  ($x_{0,2}$) and busy ($y_1$) servers.
Also, Figure~\ref{fig:simulations} (right) plots the average number of jobs per warm server, which in state~$x$ is given by $Q(x)\bydef\frac{1}{y_0}\sum_{i\ge 1} ix_{i,2}$.
The fluid (stochastic) trajectories are always represented by dashed (continuous) lines and each curve is the average of ten simulations.
Each simulation is based on~$10^6$ events.
We have set $\lambda=0.7$, $\alpha=0.05$, $\beta=0.1$ and $\gamma=0.025$.
%
%
As scaling rule, we have chosen Blind-$\theta$ where $\theta=\frac{0.5}{\alpha}\frac{1-\lambda}{ \frac{1}{\beta}  + \frac{1}{\gamma}  }$;
this choice will ensure that a strictly positive proportion of cold servers exists in the long run (see Theorem~\ref{th2}).
As dispatching algorithm, we have used Power-of-$2$ (for JIQ, see Section~\ref{sec:numerical}).
At time zero, we have assumed that the system is dimensioned exactly for the average demand, i.e., $(1-\lambda) N$ servers are cold and the remaining ones are idle-on.
\begin{figure*}
\centering
%
%
\makebox[\textwidth][c]{\hspace{0cm}\includegraphics[width=1.1\textwidth,height=7.7cm]{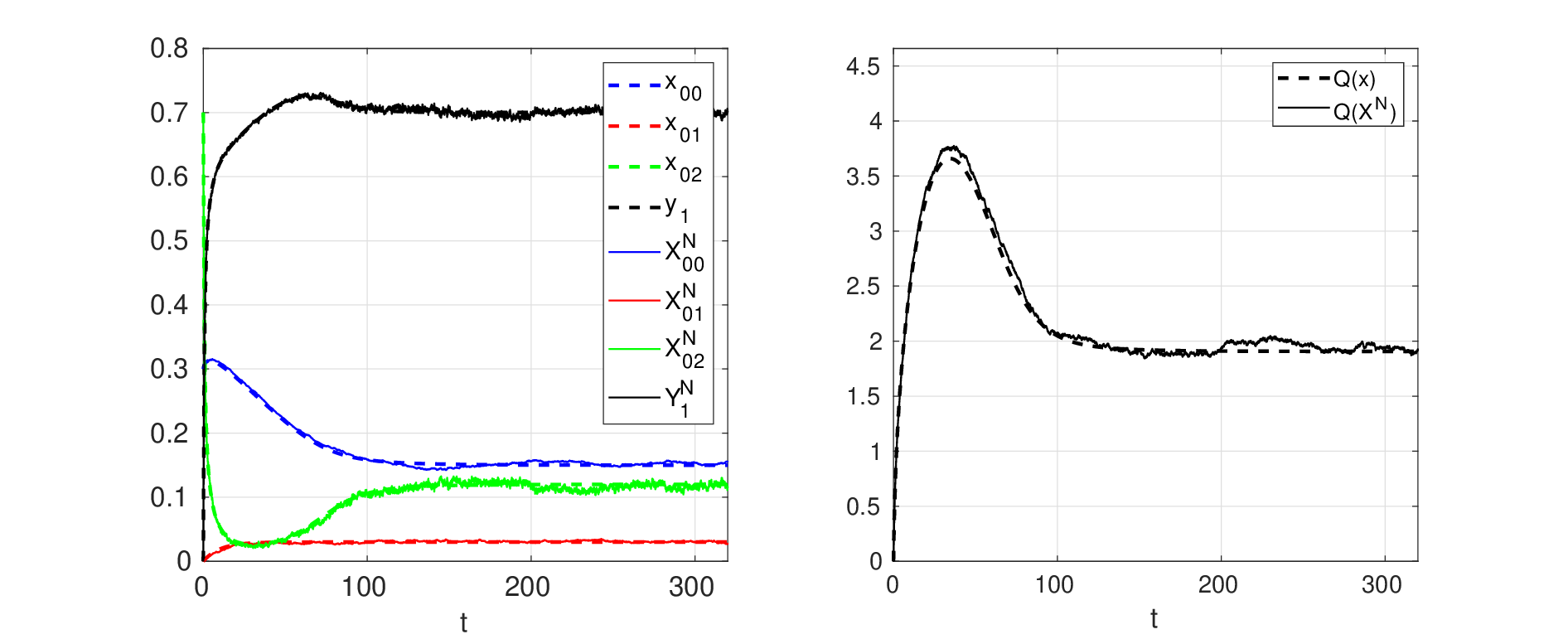}}%
\caption{Numerical convergence of the stochastic model $X^N(t)$ (continuous lines), $N=10^3$, to the fluid model~$x(t)$ (dashed lines)
when combining Power-of-2 and Blind-$\theta$.
}
%
%
%
%
%
\label{fig:simulations}
\end{figure*}
In both pictures, we observe that the fluid model captures the dynamics of $X^N(t)$ accurately.

Let us comment on the dynamics in Figure~\ref{fig:simulations}.
Initially, the system is close to instability as capacity exactly matches demand.
Here,
$Q(x(t))$ increases rapidly and
as soon as a warm server is created, it is filled with a job and as a result the proportion of idle-on servers decreases.
These decrease also because
they are not discovered fast enough upon job dispatching, thus letting them go cold even in heavy load.
This explains why the number of cold servers (the blues lines) is increasing at the beginning.
Then, more warm servers are created to mitigate the effect of the ``close to instability'' window on the accumulated overall number of jobs. Here, the mass of busy servers ($y_1$) becomes greater than the average demand $\lambda=0.7$ and $Q(x(t))$ decreases.
Finally, dynamics stabilize and in equilibrium there is a strictly positive fraction of servers that remain cold, initializing and idle-on. This indicates that there is a flux of idle-on servers that expires continuously even in equilibrium.

\subsection{Characterization of Fixed Points}
\label{sec:fixed_points}

The fluid model has the form $\dot{x}=F(x)$; see Definition~\ref{def1}.
We say that $x^*\in\mathcal{S}_1$
is a \emph{fixed point} if~$F(x^*)=0$.
We now investigate the fixed points of fluid model when buffer sizes are infinite and $\lambda$ is constant and less than one (for stability).

%
%
%
%
Let us define the following conditions:
\begin{subequations}
\label{FP_PoD_conditions}
\begin{align}
\label{FP_PoD_conditions1}
&x_{0,0} + x_{0,1} + x_{0,2} +\lambda=1 \\
\label{FP_PoD_conditions2}
&\beta x_{0,1}  =\gamma x_{0,2}\\
\label{FP_PoD_conditions3}
&\gamma x_{0,2}  \le  \alpha g(x),\quad  \mbox{ if } x_{0,0}=0 \\
\label{FP_PoD_conditions4}
&\gamma x_{0,2}  =  \alpha g(x),\quad  \mbox{ if } x_{0,0}>0
%
%
%
%
%
%
\end{align}
\end{subequations}
and if Power-of-$d$ is used:
\begin{align}
\label{FP_PoD_conditions5}
x_{i,2}  =
(\lambda+x_{0,2})\left(
\left( \tfrac{\lambda}{\lambda+x_{0,2}} \right)^{\frac{d^i-1}{d-1}}
-
\left( \tfrac{\lambda}{\lambda+x_{0,2}} \right)^{\frac{d^{i+1}-1}{d-1}}
\right),
\end{align}
for all $i\ge 1$,
otherwise if JBT-$d$ is used:
\begin{subequations}
\label{FP_PoD_conditions7_FULL}
\begin{align}
\label{FP_PoD_conditions7}
 \mbox{if }x_{0,2}=0:&\,
x_{i,2}  = 0,\quad 0\le i\le d\\
\label{FP_PoD_conditions7aa}
&\, x_{d+i,2}  = x_{d+1,2}  \left(1 - \frac{{x}_{d+1,2}}{\lambda}  \right)^{i-1},\, i\ge 2\\
\label{FP_PoD_conditions7a}
&\, x_{d+1,2} \in (0, \lambda]\\
\label{FP_PoD_conditions7b}
&\, g(x)=0\\
%
%
%
\label{FP_PoD_conditions6}
\mbox{if }x_{0,2}>0:&\,  x_{i,2} = \left( \frac{ \lambda}{z_d+x_{0,2}}\right)^i x_{0,2}\, \I{1\le i\le d+1},\, i\ge 1
\end{align}
\end{subequations}
with
$z_d\in[0,1]$ being the unique solution of
\begin{align}
\label{z_d_FP}
z_d+x_{0,2} = \frac{1-\left( \frac{ \lambda}{z_d+x_{0,2}}\right)^{d+1}}{1-\frac{ \lambda}{z_d+x_{0,2}} }   x_{0,2}
\end{align}
if $d\ge 1$ and $z_d=0$ if $d=0$.
Here, $z_d$ is interpreted as the proportion of  busy servers with no more than $d$ jobs.
%

Now, let us also introduce the following assumption, {which we will only use in Theorem~\ref{th2} below}.
\begin{assumption}
\label{as4}
For any $x_{0,2}\in[0,1-\lambda]$,
\eqref{FP_PoD_conditions5}-\eqref{FP_PoD_conditions6} uniquely determine $x_{i,2}$ for all $i\ge 1$.
\end{assumption}
Within Power-of-$d$, this assumption is clearly satisfied by
\eqref{FP_PoD_conditions5}.
Within JBT-$d$, it is satisfied only if $x_{0,2}>0$, as if $x_{0,2}=0$, then $x_{d+1,2}$ is only required to belong to $(0,\lambda]$.
Under Assumption~\ref{as4},
let $x^\circ=(x_{i,j}^\circ)$ be the unique point in $\mathcal{S}_1$ such that $x_{0,0}^\circ=0$,
$x_{0,1}^\circ  =  \frac{\gamma}{\beta+\gamma}(1-\lambda) $,
$x_{0,2}^\circ  =  \frac{\beta}{\beta+\gamma}(1-\lambda) $.

The following result characterizes fixed points.

\begin{theorem}
\label{th2}
Assume that $\lambda$ is constant and less than one.
If $x^*$ satisfies the conditions in \eqref{FP_PoD_conditions}-\eqref{z_d_FP}, then it is a fixed point of the fluid model with $B=+\infty$.
In addition,
under Assumption~\ref{as4}
\begin{enumerate}
 \item If
\begin{align}
\label{saturated_condition}
\alpha g(x^\circ) < \frac{1-\lambda}{ \frac{1}{\beta}  + \frac{1}{\gamma}  },
\end{align}
then $x_{0,0}^*>0$.
\item If \eqref{saturated_condition} does not hold, then $x^*=x^\circ$ is the unique fixed point.
\end{enumerate}

\end{theorem}
\begin{proof}
Given in Appendix~1.
\end{proof}

At the fluid scale and in a fixed point, Theorem~\ref{th2} also provides the boundary scaling probability that distinguishes between a ``saturated'' and a non-saturated system.
%
Specifically,
%
if the scaling rule satisfies \eqref{saturated_condition},
then in a fixed point there exists a fraction of idle-on servers that go cold and instantly become initializing, provided that $g(x^\circ)>0$. Here,
the pool of cold servers remains non-empty.
%
On the other hand,
if $g(x^\circ)$ does not satisfy \eqref{saturated_condition},
then no cold server exists in a fixed point but we observe that
\eqref{FP_PoD_conditions1}
and
\eqref{FP_PoD_conditions2}
imply that a strictly positive fraction of servers remain initializing, i.e., $\frac{\gamma}{\beta+\gamma}(1-\lambda)$.
Here, the interpretation is that there still exists a mass of idle-on servers that go cold but instantly become initializing while keeping the proportion of cold servers down to zero.
This corresponds to a waste of resources because initializing servers cannot process jobs.
In other words, a better performance may be obtained by keeping the initializing servers warm at all times (no auto-scaling); recall also that billing policies charge warm and initializing servers.


Within Blind-$\theta$, $g(x)=\theta$ and the conditions~\eqref{FP_PoD_conditions}-\eqref{z_d_FP} easily identify a unique fixed point, say $x^*$, with $(x_{0,0}^*,x_{0,1}^*,x_{0,2}^*)$ \emph{not} depending on the choice of the load balancing algorithm.

The following remark says that uniqueness is not always guaranteed.
\begin{remark}[Multiple Fixed Points]
\label{sc8as9s}
Suppose that $g(x)=0$ whenever $y_{1}=\lambda=1-x_{0,0}$ and that JBT-$d$ is used.
Then, Theorem~\ref{th2} implies that uncountably many fixed points exist.
In fact,
while $x_{i,2}=0$ for all $i=0,\ldots,d$ and $x_{i,2}$ is uniquely determined for all $i\ge d+2$ \emph{once fixed} $x_{d+1,2}$,
the conditions \eqref{FP_PoD_conditions}-\eqref{z_d_FP}
do not tie $x_{d+1,2}\in(0,\lambda]$ to a specific value.
%
%
%
%
\end{remark}


\subsubsection{Blind-$\theta$ and Random Dispatching}

For illustration purposes, let us consider Blind-$\theta$ with random dispatching (Power-of-$1$).
This combination does not involve any communication overhead among the auto-scaler, dispatchers and servers, and for this reason it is well suited for large systems with vast numbers of dispatchers.
Here, Theorem~\ref{th2} identifies a unique fixed point, $x^*$.
After some algebra, we obtain
$x_{0,2}^* = \min \left\{
\frac{\alpha\theta}{\gamma},\,
x_{0,2}^\circ
\right\}
$
and for the mean queue length per warm server, $Q(x)=\frac{1}{y_0}\sum_i i x_{i,2}$, we obtain (using also \eqref{FP_PoD_conditions5})
\begin{align}
\label{c9as0cs}
Q(x^*)
 =  \frac{\lambda}{\min \left\{
\frac{\alpha\theta}{\gamma},\,
x_{0,2}^\circ
\right\}}.
\end{align}
As long as a strictly positive fraction of cold servers exists,
or equivalently $\frac{\alpha\theta}{\gamma} <x_{0,2}^\circ$, we remark that $Q(x^*)$ grows \emph{linearly} in~$\lambda$.

\subsection{Optimal Design}
\label{sec:optimal_design}

Within Blind-$\theta$, Theorem~\ref{th2} guarantees the existence of a unique fixed point and
all of our numerical simulations, which we omit, indicate that it is a global attractor.
Here, necessarily $x_{0,2}^*>0$, by \eqref{FP_PoD_conditions4}, which means that a number of warm servers remain idle-on in equilibrium. Clearly, this is not optimal for energy consumption because idle-on servers consume energy.
%
%
%
%
Our goal now is to design scaling rules ensuring that a global attractor exists \emph{and} given by $x^\star$, where $x^\star\in\mathcal{S}$ is uniquely defined by $x_{0,0}^\star=1-\lambda$ and $x_{1,2}^\star=\lambda$.

\begin{remark}[Fluid Optimality]
\label{rem:asymptotic_optimality}
In $x^\star$ dynamics have achieved \emph{``delay and relative energy optimality''} in the sense that both the waiting time of jobs and the relative energy
portion consumed by idle-on and initializing servers vanish in the limit.
%
%
Here, a possible intuition is that each job is always assigned to a busy server with exactly one job but at the precise moment where it completes the processing of its previous job.
Therefore, service capacity perfectly matches demand.
\end{remark}

A direct consequence of Theorem~\ref{th2} and \eqref{FP_PoD_conditions4} is that it is necessary to impose~$g(x^\star)=0$ to achieve fluid optimality.
Within Power-of-$d$, this is impossible
as
this condition would imply that
$x_{0,2}=0$, and then \eqref{FP_PoD_conditions5} would imply~$x_{i,2}=0$ for all $i$, contradicting that~$\|x\|=1$.
In fact, Theorem~\ref{th2} implies that the unique candidate is JIQ, though it leaves open the possibility that $x(t)$ may converge to a fixed point in the sub-optimal set $\SO$, see \eqref{eq:SO}.
%
%
%
Thus, it remains to understand what additional structure the scaling rule~$g(x)$ should satisfy to make $x^\star$ a global attractor.
%
Here,
Remark~\ref{sc8as9s} suggests that even the knowledge of the amount of busy servers is not enough. More precisely, it implies that one needs $g(x)>0$ for all $x\in \SO$ as otherwise multiple fixed points exist. Therefore, given the structure of $\SO$ and $x^\star$, we have the following remark.

\begin{remark}
\label{rem:x12_necessary}
A fluid optimal scaling rule needs the access to the amount of busy servers with exactly one job, i.e.,~$x_{1,2}$.
\end{remark}

%
%

The following result provides a general condition that yields fluid optimality.


\begin{theorem}[Optimal Design]
\label{th3}
Let $\beta<1$ and
let $x(t)$, with $x(0)\in\mathcal{S}_1$, denote a fluid solution
induced by JIQ and any scaling rule $g(x)$
that satisfies, beyond Assumption~\ref{as3},
\begin{align}
\label{gx_optimal}
g(x) = 0 ~{\rm if~and~only~if} ~  x_{1,2} + \beta x_{0,1} \ge \lambda.
\end{align}
Then, $\lim_{t\to\infty}\| x(t) - x^\star \|_w  = 0$.
\end{theorem}
\begin{proof}
 Given in Appendix~1.
\end{proof}


The interpretation
is that $x_{1,2} + \beta x_{0,1} $ represents the overall rate at which servers become idle-on.
Thus, our optimality condition  says  to scale up resources whenever
the excess of the mean demand over the rate at which servers become idle-on is positive, as
in this case JIQ is smart enough to fill them up immediately saturating the surplus service capacity.
Otherwise, if the excess is negative, one can turn the scale-up process  off ($g=0$), and in this case the natural dynamics induced by both JIQ and the scale-down rule are enough to drive the system behavior to the desirable configuration~$x^\star$.


\begin{remark}
\label{rem:optim}
%
\emph{Rate-Idle}, see Definition~\ref{as3}, satisfies \eqref{gx_optimal}.
If $g$ denotes \emph{Rate-Idle} and $f:[0,1]\to[0,1]$ is continuous, onto and increasing,
then $f(g)$ is a scaling rule that as well satisfies~\eqref{gx_optimal}.
\end{remark}

As discussed in Section~\ref{sec:sys_desc}, the assumption $\beta<1$, i.e., the mean server initialization rate is smaller than the mean job service rate, is largely accepted in practice~\cite{scaleperrequest,Peeking18}. From a mathematical standpoint, it is not necessary for fluid optimality but simplifies our proof.

\begin{remark}[Communication Overhead]
\label{rem:overhead}
A scaling rule satisfying~\eqref{gx_optimal} requires the central controller
to have access to the amount of initializing and busy servers containing exactly one~job, i.e., $x_{0,1}$ and $x_{1,2}$. Since an initializing server informs the platform as soon as it becomes warm, $x_{0,1}$ is easily obtained in practice.
For $x_{1,2}$, the auto-scaler can run a local memory with $N$ slots, where the $n$-th slot indicates the state of server~$n$, say `Cold',  `Init',  `Idle-on',  `Busy$_1$' and `Busy$_{\ge 2}$', with obvious interpretations.
Then, one way to update the memory is by letting each server send a message to the auto-scaler whenever the transitions
`Busy$_{\ge 2}$' $\to$ `Busy$_{1}$,
`Busy$_1$' $\to$ `Idle-on' and
`Idle-on' $\to$ `Busy$_1$' occur.
As in standard implementations of JIQ, this involves only a constant number of messages per job to be exchanged between the auto-scaler and the servers.
%
%
%
%
\end{remark}


\subsection{Convergence to Multiple Fixed Points}
\label{sec:convergence_multiple}



{
In Theorem~\ref{th3}, we have provided a condition ensuring that $x^\star$ is globally stable.
In this section, we show that it is not always possible to have global stability.
To guarantee stability, one may expect that is enough to have a strictly positive scaling probability whenever the current capacity of warm servers is less than the average demand, i.e., $g(x)>0$ whenever $y_0<\lambda$.
The following proposition shows that this intuition is false.
}

%

Let
\begin{multline}
\label{eq:SO}
\SO \bydef \Big\{x\in\mathcal{S}:
x_{0,0}=1-\lambda,\,  x_{0,1}=x_{0,2}=0,\\ x_{1,2}<\lambda \mbox{ and \eqref{FP_PoD_conditions7aa} holds  with } d=0 \Big\}
\end{multline}
and
let $\overline{Q}(x)\bydef \sum_{i\ge 1}ix_{i,2}$ denote the average number of jobs per server in state~$x\in\mathcal{S}$; here, cold and initializing servers are included in the counting.

\begin{proposition}
\label{th4}
Assume that $\lambda$ is constant and less than one.
Let $g(x)$ be any scaling rule such that
\begin{align}
g(x)=\frac{1}{\lambda}(x_{0,0}-1+\lambda)^+,\qquad \forall x\in\mathcal{S}: y_0<\lambda.
\end{align}
Let $x(t)$ denote a fluid model induced by such $g(x)$ and JIQ such that
%
\begin{multline}
\label{uc8a9gtg}
x_{0,0}(0)>1-\lambda,\,\,
x_{0,2}(0)=0,\\
x_{1,2}(0)+\beta x_{0,1}(0)< \lambda < \overline{Q}(x(0))<\infty.
\end{multline}
Suppose that $\beta <1$, $\alpha\neq\beta$ and $B=+\infty$.
Then,
\begin{align}
\label{g_positive}
g(x(t))&
> 0, \quad \forall t\ge 0
%
%
\end{align}
\begin{multline}
\label{Qlimit}
\lim_{t\to\infty} \overline{Q}(x(t)) = \overline{Q}(x(0))
  +  \frac{x_{0,1}(0)}{\beta}
\\+  \frac{\alpha+\beta}{\alpha\beta} (x_{0,0}(0)-1+\lambda)
>\lambda.
\end{multline}
In addition,
%
%
$y_0(t)\uparrow\lambda$, and if $x_{1,2}(t)\to x_{1,2}(\infty)$,  then $x(t)\to x(\infty)$ with~$x(\infty)\in\SO$.
\end{proposition}
\begin{proof}
Given in Appendix~3.
\end{proof}




Thus, while the proportion of warm servers converges to~$\lambda$,
such convergence may occur \emph{from below} even if there always exists a strictly positive probability of creating new warm servers.
In this case, the average demand is greater than the current service capacity at any point in time and
this makes the mean queue length converge to a limit that depends on the initial conditions.
%
%

Let us comment a little bit further and prepare the setting for our next contribution.
To create the underload situation above where $y_0(t)\uparrow\lambda$, it is not necessary to assume that all warm servers are initially busy ($x_{0,2}(0)=0$), though we have included this condition in \eqref{uc8a9gtg} to simplify our proof.
In contrast, to avoid this situation, it may be sufficient that $g(x)$ is bounded away from zero whenever $y_0<\lambda$.
By continuity, this implies that $g(x)>0$ as well whenever $y_0=\lambda$, but in this case the resulting scaling rule will not possess the optimality property stated in Theorem~\ref{th3} below (as this will imply that~$g(x^\star)>0$).
On the other hand, one may consider a scaling rule that is discontinuous on the set $\{x:y_1=\lambda\}$, a setting that does not satisfy Assumption~\ref{as3}.
Here, beyond revisiting Theorem~\ref{th1} for justification of the fluid model, the problem is that scale-up decisions would significantly depend on small perturbations of the equilibrium system state, severely impacting robustness from a practical standpoint.



%

\section{Empirical Comparison: Synchronous vs Asynchronous}
\label{sec:empirical}

The structural differences between the synchronous and asynchronous approaches have been described in Section~\ref{sec:sync_vs_async}.
{
In this section, we compare both approaches by means of numerical simulations.
Specifically, we compare
our asynchronous combination of JIQ and Rate-Idle (see Definition~\ref{blind})
with a generalization of TABS, i.e., the synchronous scheme developed in~\cite{elasticSIG}.
For the latter, we assume that $d$ servers are initialized at the moment of a job arrival if all active servers are busy upon arrival of that job, in which case the job is sent to a (busy) server at random.
Thus, the TABS scheme in~\cite{elasticSIG} is recovered when $d=1$.
Let us refer to such generalization as TABS-$d$.
Clearly, $d$ affects the scale-up rate and plays the same role of $\alpha$ in ALBA.
To make the comparison fair, we will assume that $\alpha$ is fine-tuned such that the resulting \emph{scale-up rate} induced by ALBA matches the scale-up rate induced by TABS-$d$; thus, $\alpha=\alpha(d)$.
Here, the scale-up rate is defined as the number of server initialization signals divided by the time horizon.
}

{
Our comparison metrics are
\begin{itemize}
 \item the empirical probability of waiting, that is the average fraction of jobs that are sent to a busy server. We refer to these as $p_{{\rm Wait}}^{\rm ALBA}$ and $p_{{\rm Wait}}^{{\rm TABS-}d}$.

\item the empirical energy consumption, that is $E = N(w_{\rm{init}} x_{0,1}(t) + w_{\rm{idle-on}}x_{0,2}(t) + w_{\rm{busy}}) y_1(t)$ averaged over time; here, we assume $w_{\rm{init}}=2$, $w_{\rm{idle-on}}=0.5$ and $w_{\rm{busy}}=1$. We refer to these as $E^{\rm ALBA} $ and $E^{{\rm TABS-}d}$.

\end{itemize}
Then, we consider the ratios
\begin{equation}
\label{comp_ratio}
\mathcal{R}_{{\rm Wait}} :=
\frac
{
p_{{\rm Wait}}^{\rm ALBA} }{p_{{\rm Wait}}^{{\rm TABS-}d}
},
\quad
\mathcal{R}_{{\rm Energy}} :=
\frac
{
E^{\rm ALBA} }{E^{{\rm TABS-}d}
},
\end{equation}
and evaluate them by simulation of $10^7$ events (both schemes have been tested within the same seed sequences)
and when
$N\in\{100,500,1000\}$, $\lambda\in\{0.35,0.7\}$, $d=\{1,5,10\}$, $\beta=0.1$ and $\gamma=0.025$.
If a time unit is 10 milliseconds, these parameters are realistic \cite{scaleperrequest,Google,Peeking18}.
We also assume that the initial condition is $x^\star$, i.e., the global attractor of the fluid dynamics defined in Section~\ref{sec:optimal_design}.
This choice measures the perturbations of order $1/N$ that appear around $x^{\star}$, which are not visible at the fluid scale.
Within this setting, Figure~\ref{fig:comparison} plots $\mathcal{R}_{{\rm Wait}}$ (blue) and $\mathcal{R}_{{\rm Energy}}$ (red)
and
shows that ALBA always provides a much smaller probability of waiting than TABS-$d$ while inducing the same energy consumption cost as $\mathcal{R}_{{\rm Energy}}$ is almost one; see the Appendix for a table containing numerical data.
In addition, this behavior is amplified when~$N$ and~$d$ increase.
\begin{figure}
\centering
\makebox[1.1\columnwidth][c]{\hspace{-0.5cm}\includegraphics[width=1.1\columnwidth]{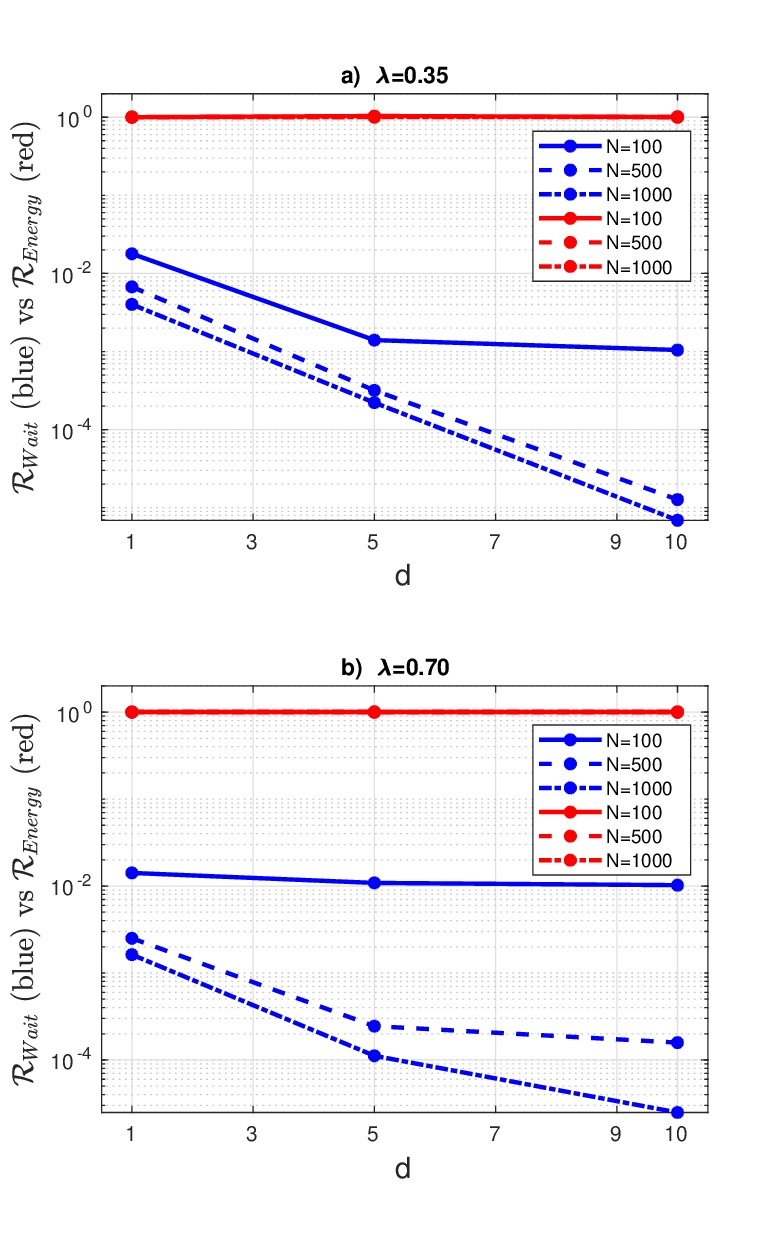}}%
%
%
\vspace{-1cm}
\caption{Ratio $\mathcal{R}^{(N)}$ of the transient probability of waiting induced by the proposed asynchronous scheme (Rate-Idle+JIQ) and the synchronous approach in \cite{elasticSIG}, respectively. The initial condition is the global attractor~$x^\star$~(defined in Section~\ref{sec:optimal_design}, see Theorem~\ref{th3}), which corresponds to delay and relative energy optimality.
}
\label{fig:comparison}
\end{figure}
}
%
%
%
%
%
As discussed in Section~\ref{sec:sync_vs_async}, we own the performance gain of ALBA to the fact that scale up decisions may be taken before job arrivals, while in a synchronous scheme such as TABS-$d$, jobs are forced to wait any time a scale up decision is taken.
{While this anticipation induces a slightly increased energy cost, it pays off because $\mathcal{R}_{{\rm Energy}}$ remains very close to one.}

{Since $\mathcal{R}_{{\rm Wait}}$ decreases with the system size~$N$, we may postulate that it approaches zero as $N\to\infty$. This requires a second-order limit analysis of the underlying Markov chains,
which we leave as future work.}

\section{Energy Optimization with Performance Guarantees}
\label{sec:numerical}


In this section, we use the fluid model in an optimization framework to trade off between performance and energy costs, and we numerically show that it accurately captures the stochastic dynamics of the finite ALBA system.
%
%
%
%
%
Let us focus on JIQ as load balancing algorithm and on the set (say $\mathcal{G}$) of scaling rules that satisfy the assumptions in Theorem~\ref{th3}. Note that these imply fluid optimality in the stationary regime.
As in, e.g., \cite{WiermanSpeedScaling}, let us define the cost function~$\mathcal{J}_{g}$ as the long-run time average of a linear combination between the power consumption $P(x)=c_{0,1} x_{0,1} + c_{0,2} x_{0,2} + c_{1,2} y_{1}$, $c_{i,j}>0$, and the average queue length per busy server $Q(x)=\frac{1}{y_1}\sum_i i x_{i,2}$ induced by the scaling rule~$g$, i.e.,
\begin{align}
\mathcal{J}_g \bydef \lim_{T\to\infty}  \frac{1}{T}\int_0^T \left( \kappa_1 P(x(t)) + \kappa_2 Q(x(t)) \right) {\rm d} t
\end{align}
where $\kappa_i\ge 0$, $i=1,2$;
one can think $\kappa_1$ in terms of \$/watt and $\kappa_2$ in terms of \$/job.
Then, Theorem~\ref{th3} implies that
\begin{align}
\inf_g \mathcal{J}_g = \mathcal{J}_{g^*} = \kappa_1 P(x^\star)+\kappa_2 Q(x^\star) = \kappa_1 c_{1,2}\lambda  + \kappa_2
\end{align}
for all $g^*\in \mathcal{G}$.
%
%
While all policies in $\mathcal{G}$ yield the same (optimal) cost,
their behavior is clearly different trajectory-wise. Depending on the application, a platform user has several options to single out a policy in $\mathcal{G}$ that satisfies a further level of optimization.
For instance, a substantial portion of the applications hosted in cloud networks have ultra-low delay requirements, as this may have important consequences on e-commerce sales.
On the other hand, also energy bills are equally important from both financial and environmental standpoints.
Here, a system manager may want to look for a scaling rule in~$\mathcal{G}$ such that
\begin{equation}
\label{Q_constraint}
Q(x(t))\le q,\quad \forall t\ge 0
\end{equation}
where~$q$ is related to the desired user-perceived performance guarantee;
by Little's law, \eqref{Q_constraint} is equivalent to a constraint on the mean response time.
In view of Remark~\ref{rem:optim}, one may consider the parameterized subset of scaling rules
\begin{align}
\label{gx_eta}
g(x) = \frac{1-\exp(-\frac{\eta}{\lambda} (\lambda - x_{1,2} -\beta x_{0,1})^+)}{1-\exp(-\eta)},\quad \eta>0,
%
%
%
%
\end{align}
which satisfy both Assumption~\ref{as3} and~\eqref{gx_optimal}.
Here, the control parameter $\eta>0$ indicates how aggressive the scaling rule is:
Rate-Idle is recovered when $\eta\downarrow 0$
and $g(x)=\I{\lambda \ge x_{1,2} +\beta x_{0,1}}$ when $\eta\to\infty$.
%
%
%
Then, one may search for the smallest (least aggressive) $\eta$ such that \eqref{Q_constraint} holds true.

The above problem can be easily addressed numerically within the proposed deterministic model.
Assume that the system is currently in a light-load condition, say $\lambda=0.25$, and that, as a result, it is dimensioned accordingly to save energy, say $x_{0,0}=1-\lambda-0.05$, with $x_{0,2}=0.05$ and $x_{1,2}=\lambda$; the extra 0.05 is meant to keep a reserve of idle-on servers ready to go.
Then, at time zero, an unexpected workload peak
occurs, and $\lambda=0.5$. Here, the platform needs to automatically adjust the service capacity while ensuring~\eqref{Q_constraint}.
Let us assume $q=2$,
$\alpha  = 0.35$,
$\beta   = 0.1$ and
$\gamma  = 0.025$.
The dashed lines in Figure~\ref{fig:Qle2} represent the
dynamics of the fluid queue lengths $Q(x(t))$ and scaling probabilities $g(x(t))$, for~$\eta=1,10^3$.
The corresponding continuous lines represent the average of ten simulations of the stochastic model $X^N(t)$ with $N=1000$.
%
%
%
%
%
\begin{figure*}
\centering
\includegraphics[width=19cm]{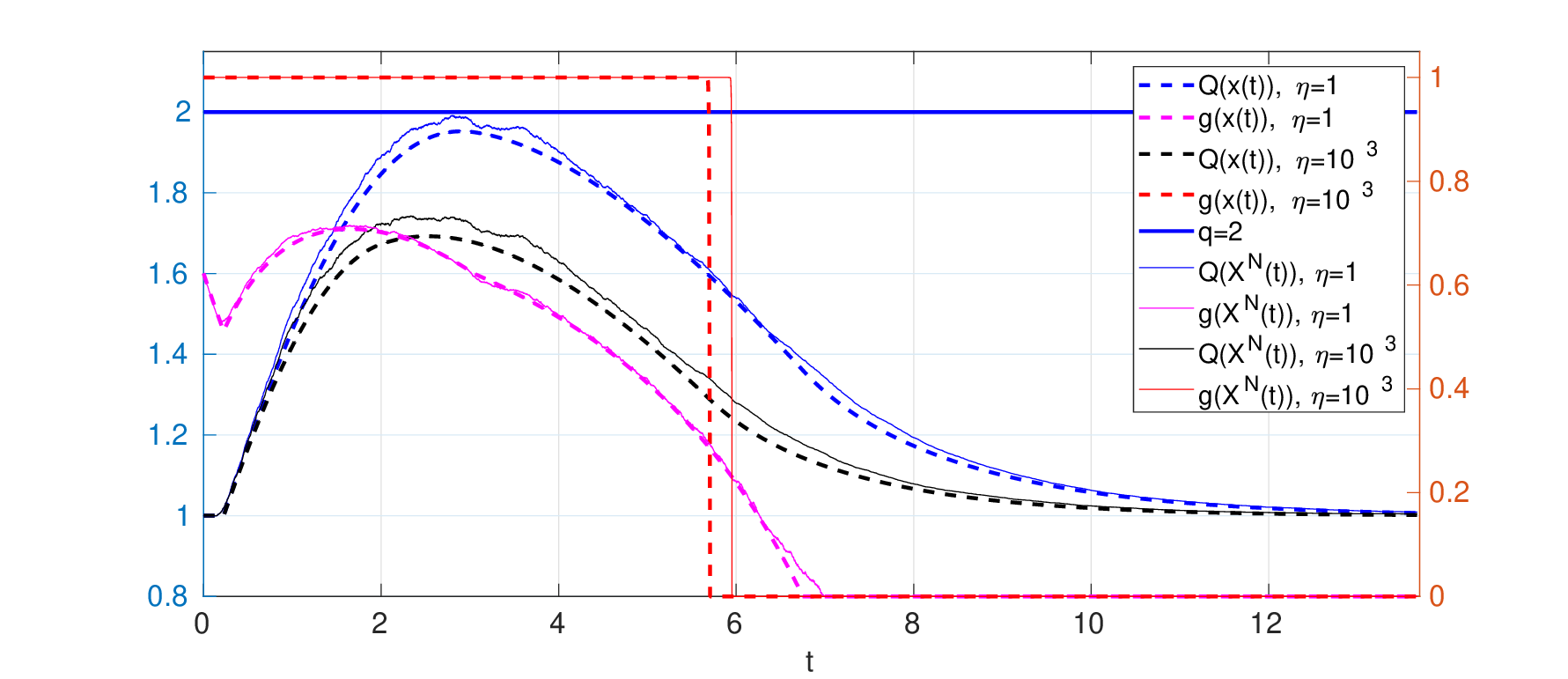}
%
%
\caption{Transient behavior of the queue lengths ($y$-axis on the left) and scaling probabilities ($y$-axis on the right) by varying~$\eta$, see~\eqref{gx_eta}, for both the fluid ($x(t)$) and stochastic ($X^N(t)$) models with $N=1000$.}
\label{fig:Qle2}
\end{figure*}
First, let us remark that the fluid model approximation accurately captures the dynamics of $X^N(t)$, though it slightly underestimates queue lengths and scaling probabilities.
Now, let us consider $\eta=1$.
Initially, queue lengths increase as expected due to the surge of demand and the scaling probability is large enough to drive the proportion of cold servers to zero. This explains the non-differentiability point of the trajectory of the scaling rule because the amount of initializing servers stops to grow.
Then,
the system has enough capacity to drain the load and at some point the rate at which servers become idle-on overflows the mean demand, i.e., $x_{1,2} +\beta x_{0,1}>\lambda$, so that eventually $g(x(t))=0$.
Finally, queue lengths assess to their asymptotic value $Q(x^\star)=1$
We conclude that $\eta=1$ is enough to make~\eqref{Q_constraint} holds true.
We also notice that the choice $\eta=10^3$, which essentially means to scale up resources at the maximum available rate $\alpha$ whenever $x_{1,2} +\beta x_{0,1}<\lambda$, has little impact on performance.
Nonetheless,
it should be clear that the larger the value of $\eta$, the larger the resulting time-average power consumption.



\section{Conclusion}
\label{sec:conclusions}

In cloud systems, load balancing and auto-scaling are key mechanisms to optimize both delay performance and energy consumption.
The focus of the existing literature has been on architectures where these mechanisms are synchronous or rely on a central queue. The novelty of our work is to consider an asynchronous and decentralized architecture.
Decentralization increases scalability and asynchronism does not force jobs to wait any time a scale-up decision is taken.

Our work provides a tractable framework to evaluate the performance of auto-scaling algorithms that are up to the platform user to design.
In our main result, we have identified a structural condition for asymptotic optimality that provides the platform user with some flexibility when designing an optimal scaling rule; see Remark~\ref{rem:optim}.
This can be exploited to develop new levels of optimization as we have shown in Section~\ref{sec:numerical}.
{By means of numerical simulations, we have show that the proposed asynchronous combination of JIQ and Rate-Idle provides a better delay performance than existing synchronous decentralized schemes while inducing almost the same energy consumption.}

%
%
%
%

We discuss some generalizations and {open questions}:
 \begin{itemize}


\item We have assumed that only one server at a time can be activated at each scaling time.
Our approach generalizes trivially to the case where a random number~$C$ of cold servers is selected, provided that the distribution of~$C$ does not depend on~$N$.
Mutatis mutandis, it is enough to replace $\alpha$ by $\alpha\, \E[C]$.

\item {Theorem~\ref{th1} generalizes trivially to a time-varying arrival rate setting if the arrival rate takes the form $\Lambda(t) N$ where $\Lambda(t)$ is a bounded positive real-valued function independent of~$N$.
This change only affects Lemma~1 of the supplementary material, whose proof directly generalizes by the functional strong law of large numbers for the Poisson process.
The resulting deterministic model is identical to the one in Definition~\ref{def1} except that $\lambda$ is replaced by $\Lambda(t)$.}


\item {From a theoretical point of view, it is interesting to prove the ``interchange of limits'' property.
More specifically, within JIQ and the asymptotically optimal condition identified in Theorem~\ref{th3}, the question is whether or not the invariant distribution of the underlying Markov chain concentrates on $x^\star$ when $N\to\infty$.
Numerical evidence indicates that this property holds true.
}



\item \review{The stability of the (finite) stochastic model is a difficult question to answer because the proposed ALBA framework is very general: the scale-up rule~$g$ satisfies mild conditions (see Assumption~\ref{as3}) and to come up with a stability result, one should take additional assumptions such as considering a specific scale-up policy.
Even within the simplest scale-up policy, i.e., Blind-$\theta$, and the simplest dispatching policy, i.e., where jobs are distributed to servers uniformly at random (or equivalently Power-of-$d$ with $d=1$), understanding whether or not the underlying Markov chain is positive recurrent is challenging.
Here, one may check that (natural adaptations of) classical Lyapunov functions used in queueing theory to investigate stability via Foster-Lyapunov theorem do not work.
Also, the utilization of  Dai's fluid framework~\cite{dai1995positive} is again complicated by the identification of a Lyapunov function.
Finally, the drift function does not preserve monotonicity and stochastic dominance arguments
cannot be applied.
}

\end{itemize}

\bibliographystyle{IEEEtran}

%



%

\begin{IEEEbiography}[{\includegraphics[width=1in,height=1in,clip,keepaspectratio]{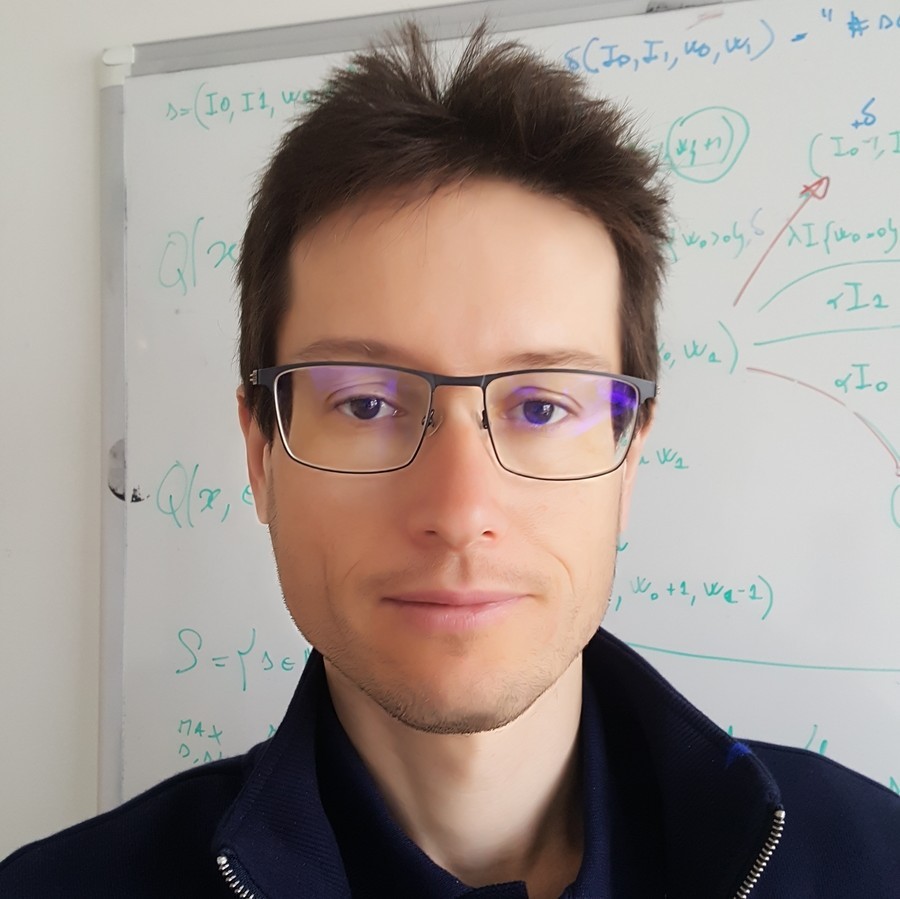}}]{Jonatha Anselmi}
is a tenured researcher at the French National Institute for Research in Digital Science and Technology (Inria), since 2014. Prior
to this, he was a researcher at the Basque Center for Applied
Mathematics and a postdoctoral researcher at Inria. He received his PhD
in computer engineering at Politecnico di Milano (Italy) in 2009. His research interests are in the broad field of decision-making under uncertainty, where computer science, applied maths and engineering intersect.
\end{IEEEbiography}

\vfill





\section{Proofs of Theorems~1,~2 and~3}
\label{sec:proofs}



\subsection{Theorem~1: connection between the fluid and the Markov models}
\label{sec:conn}

To prove Theorem~1, we follow two main steps.
First, we couple the processes $(X^{N}(t))_{t\in[0,T]}$, for all $N\in\mathbb{Z}_+$, on a common probability space and
show that limit trajectories exist and are Lipschitz continuous with probability one.
The arguments used in this step are routine~[7], [10], [28].
Then, we prove that limit trajectories are fluid solutions, which is the main technical difficulty, and here we develop arguments specific to the model under investigation.


\subsubsection{Coupled construction of sample paths}
\label{sec:constructionX}

Let $\mathcal{N}_c(t)$ denote a Poisson process of rate $c$.
We construct a probability space where the stochastic processes $\{(X^{N}(t))_{t\in[0,T]}\}_{N\ge 1}$ are coupled.
All the processes of interest can be constructed in terms of the following mutually independent primitive processes:

\begin{itemize}
\item $\mathcal{N}_{\phi} (t)$, a Poisson process of rate $\phi\bydef \lambda+1+\alpha+\beta+\gamma$.
This process is defined  on $(\Omega_E, \mathcal{A}_E, \mathbb{P}_E)$
and each jump of $\mathcal{N}_{\phi}(t)$ denotes the occurrence of an \emph{event}.

\item $(W_n)_n$, where the random variables $W_n$ are $\{0,1,2,3,4\}$-valued i.i.d. and such that
$\pp(W_n=0)=\lambda/\phi$,
$\pp(W_n=1)=1/\phi$,
$\pp(W_n=2)=\alpha/\phi$
$\pp(W_n=3)=\beta/\phi$ and
$\pp(W_n=4)=\gamma/\phi$.
This process is defined on $(\Omega_W, \mathcal{A}_W, \mathbb{P}_W)$ and will identify the \emph{type} of the $n$-th event. Specifically, $W_n=0$ indicates a job arrival, $W_n=1$ a potential job departure, $W_n=2$ a {scaling} time, $W_n=3$ a potential server initialization, i.e., a server completed the initialization phase, and $W_n=4$ a potential server expiration.

\item $(A_n^p)_n$, $p=1,\ldots,d$, $(D_n)_n$, $(I_n)_n$, $(E_n)_n$ and $(R_n)_n$, where
the random variables $A_n^p$, $D_n$, $I_n$, $E_n$ and $R_n$, for all $n$, are all i.i.d. and uniform over the interval $[0,1]$.
The rvs $A_n^p$, $D_n$, $I_n$, $E_n$ will be respectively used to select a server that
i) will process an arriving job,
ii) fires a departure,
iii) fires an initialization and
iv) fires an expiration.
The rv $R_n$ is related to the scaling rule and will decide whether a new server will be activated.
These processes are defined on $(\Omega_S, \mathcal{A}_S, \mathbb{P}_S)$;

\item $(X^N(0))_N$, the process of the initial conditions, where each random variable $X^N(0)$ takes values in~$\mathcal{S}_N$. This process is defined on $(\Omega_0, \mathcal{A}_0, \mathbb{P}_0)$.

\end{itemize}

Using that $\mathcal{N}_{\phi} (N t)$ and $\mathcal{N}_{\phi N} (t)$ are equal in distribution and the well-known fact that thinnings of a Poisson process produce independent Poisson processes, each process $\{(X^{N}(t))_{t\in[0,T]}\}$, $N\ge 1$, can be constructed on the product space, say $(\Omega,\mathcal{A},\mathbb{P})$.

Now, let $t_n$ be the time of the $n$-th jump of $\mathcal{N}_\phi(N t)$.
Let also $X^N(t^-)\bydef \lim_{s \uparrow t } X^N(s)$,
 $Y_i^N(t) \bydef \sum_{j=0}^i X_{j,2}^N(t)$ for all $i\ge 0$, $Y_{-1}^N(t)=0$
 and $\II{A}{x}=1$ if $x\in A$ and 0 otherwise.
\review{
Note that in the main text,  $y_i$ is defined as a tail sum while here $Y_i^N$ is a cumulative sum.
}
The coordinates of $X^N(t)$ are then given by~\eqref{X_construction} for all $i\ge 1$.
\begin{figure*}[!t]
\normalsize
%
%
%
%
\begin{subequations}
\label{X_construction}
\begin{align}
X_{0,0}^N(t) = &  X_{0,0}^N(0)
+ \frac{1}{N} \sum_{n=1}^{\mathcal{N}_\phi(N t)}
\left(  \I{W_n=4}
   \II{(0,X_{0,2}^N(t_n^-)]}{E_n}
- \I{W_n=2}
\I{X_{0,0}^N(t_n^{-})>0}
\, \II{(0,g(X^N(t_n^{-}))]}{R_n} \right) \\
X_{0,1}^N(t) = &  X_{0,1}^N(0)
+ \frac{1}{N} \sum_{n=1}^{\mathcal{N}_\phi(N t)}
\left(\I{W_n=2}
\I{X_{0,0}^N(t_n^{-})>0}
\, \II{(0,g(X^N(t_n^{-}))]}{R_n}
-  \I{W_n=3}
   \II{(0,X_{0,1}^N(t_n^-)]}{I_n}
   \right)
\\
X_{0,2}^N(t) =&   X_{0,2}^N(0)
+ \frac{1}{N} \sum_{n=1}^{\mathcal{N}_\phi(N t) }
\left(
  \I{{W_n=1}}\II{[Y_0^N(t_n^-),Y_1^N(t_n^-)]}{D_n}
- \I{{W_n=0}} H_{0}(t_n^{-})   +  \I{W_n=3}
   \II{(0,X_{0,1}^N(t_n^-)]}{I_n}
- \I{W_n=4}
   \II{(0,X_{0,2}^N(t_n^-)]}{E_n}
   \right)
   \\
X_{i,2}^N(t) =&  X_{i,2}^N(0)
+ \frac{1}{N} \sum_{n=1}^{\mathcal{N}_\phi(N t) }
\left(\I{W_n=0}\left( H_{i-1}(t_n^{-}) - H_{i}(t_n^{-})\I{i<B}\right)
+ \I{W_n=1}
\left(
 \II{(Y_i^N(t_n^-),Y_{i+1}^N(t_n^-)]}{D_n}
-
 \II{(Y_{i-1}^N(t_n^-),Y_i^N(t_n^-)]}{D_n}
\right)\right)
\end{align}
\end{subequations}
%
%
%
%
\hrulefill
\vspace*{4pt}
\end{figure*}
In \eqref{X_construction}, the $H_i$ terms depend on the load balancing scheme used:
within Power-of-$d$ (servers are selected with replacement)
\begin{multline}
\label{h_i_Pod_proof}
H_i (t_n^{-})
\bydef
\prod_{p=1}^d \II{(Y_{i-1}^N(t_n^-),1]}{A_n^p (1-X_{0,0}^N(t_n^-)-X_{0,1}^N(t_n^-) )} \\
-
\prod_{p=1}^d \II{(Y_{i}^N(t_n^-),1]}{A_n^p (1-X_{0,0}^N(t_n^-)-X_{0,1}^N(t_n^-) )} \in \{0,1\}
\end{multline}
and
within JBT-$d$
\begin{multline}
\label{h_i_JBTd_proof}
H_i (t_n^{-})
\bydef
\II{(Y_{i-1}^N(t_n^-),Y_{i}^N(t_n^-)]}{A_n^1 (1-X_{0,0}^N(t_n^-)-X_{0,1}^N(t_n^-) )}
\I{Y_d^N(t_n^-)=0} \\
 +
\II{(Y_{i-1}^N(t_n^-),Y_{i}^N(t_n^-)]}{A_n^1 Y_d^N(t_n^-)}
\I{i\le d}
\I{Y_d^N(t_n^-)>0} \in \{0,1\}.
\end{multline}
These expressions follow by uniformization of $X^N(t)$.
For instance,  $X_{0,0}^N(t)$ has an upward jump of size $1/N$ at time $t_n$ if the event occurring at that time is of type 4 (potential server expiration) and an idle-on server is actually selected at time $t_n^-$ by the uniformized process.
Analogously,  $X_{0,0}^N(t)$ decreases by $1/N$ at time $t_n$ if the event occurring at that time is of type~2, provided that at time $t_n^-$ the cold servers pool is not empty and the scaling rule applies.
Similar interpretations hold along the other coordinates of $X^N(t)$.


\subsubsection{Tightness of sample paths and Lipschitz property}
We now prove tightness of sample paths.
The lemmas in this section are routine and equivalent to the lemmas in~[11,Section~5.2].

Let us introduce the following formulas for quick reference.

\begin{lemma}
\label{GC_F}
Let $T>0$.
There exists $\mathcal{C}\subseteq \Omega$ such that $\mathbb{P}(\mathcal{C})=1$
and for all $\omega\in\mathcal{C}$:
\begin{equation}
\label{wqw8uw1}
\lim_{N\to\infty} \sup_{t\in[0,T]} | \frac{1}{N} \mathcal{N}_\phi(N t,\omega) -\phi t| =0\\
\end{equation}
\begin{equation}
\label{wqw8uw2}
\lim_{N\to\infty} \sup_{t\in[0,T]} \Big| \frac{1}{N} \sum_{n=1}^{\mathcal{N}_\phi(N t,\omega)} \I{W_n(\omega)=k} -  \pp(W_1=k)\,\phi \,t\,\Big| =0
\end{equation}
for all $ k\in\{0,\ldots,4\}$, and %
\begin{equation}
\label{wqw8uw3}
 \lim_{N\to\infty} \frac{1}{N} \sum_{n=1}^N
\prod_{p=1}^d \II{(a_p,b_p]}{c_pA_n^p}
= \prod_{p=1}^d \frac{b_p-a_p}{c_p}
\end{equation}
for all  $a_p,b_p,c_p\in [0,1], c_p>0, p=1,\ldots,d.$
\end{lemma}
\begin{proof}
This lemma directly follows by applying the functional strong law of large numbers for the Poisson process (for \eqref{wqw8uw1}), the fact that thinnings of a Poisson process produce independent Poisson processes (for \eqref{wqw8uw2}) and the strong law of the large numbers (for \eqref{wqw8uw3}).
\end{proof}

We will work on a fixed $\omega$ that belongs to $\mathcal{C}$.

Let $x^0\in[0,1]$, sequences $A_N\downarrow 0$ and $B_N\downarrow 0$ be given.
Let also $D[0,T]$ denote the Skorokhod space endowed with the uniform metric
$d(x,y)\bydef \sup_{t\in[0,T]} |x(t)-y(t)|$, for all~$x,y\in D [0,T]$.
For $N\ge 1$, let also
\begin{multline*}
\mathcal{E}_N(B_N, A_N,x^0)  \bydef \big\{ x\in D[0,T] :  |x(0)-x^0|\le B_N, \\
|x(a)-x(b)| \le \phi|a-b|+A_N, \,\forall a,b\in[0,T]  \big\}
\end{multline*}
\begin{multline*}
\mathcal{E}_c(x^0)  \bydef \big\{ x\in D[0,T] : x(0)=x^0,  \\
|x(a)-x(b)| \le \phi|a-b|, \,\forall a,b\in[0,T]  \big\}.
\end{multline*}

The next lemma says that the sample paths along any coordinate is approximately Lipschitz continuous.
The proof is omitted because follows exactly the same standard arguments used in Lemma~5.2 of [11], which basically use the fact that the jumps of the Markov chain of interest are of the order of $1/N$ and that the evolution of such Markov chain on a given coordinate only depends on the evolution of such Markov chain on a finite number of other coordinates.

\begin{lemma}
\label{LM:tightness}
Fix $T>0$, $\omega\in\mathcal{C}$, and some $x^0\in\mathcal{S}_1$. Suppose that $\|X^N(\omega,0)-x^0\|_w \le \tilde{B}_N$, for some sequence $\tilde{B}_N\downarrow 0$. Then, there exists sequences $\left\{B_N^{(i,j)}\downarrow 0\right\}_{i,j} $ and $A_N\downarrow 0$ such that
\begin{align}
X_{i,j}^N(\omega,\cdot)\in \mathcal{E}_N(B_N^{(i,j)}, A_N,x^0),\quad \forall (i,j), \,\forall N.
\end{align}
\end{lemma}

The next proposition shows that any sequence of sample paths $X^N(\omega,t)$ contains a further subsequence that converges in $D^\infty[0,T]$, endowed with the metric $d^{\mathbb{Z}_+}(x,y)\bydef \sup_{t\in[0,T]}\|x(t)-y(t)\|_w$, to a coordinate-wise Lipschitz continuous trajectory $x(t)$, as long as $\omega\in\mathcal{C}$.
The proof is routine and omitted because it is a repetition of the argument used in the proof of Proposition~11 in [28] (equivalently, see also Proposition~5.3 in [11]).

\begin{proposition}
\label{PR:tightness}
Fix $T>0$, $\omega\in\mathcal{C}$, and some $x^0\in\mathcal{S}_1$.
Suppose that $\|X^N(\omega,0)-x^0\|_w \le \tilde{B}_N$,
for some sequence $\tilde{B}_N\downarrow 0$. Then,
every subsequence of $\{X^N(\omega,\cdot )\}_{N=1}^\infty$ contains a further subsequence
$\{X^{N_k}(\omega,\cdot )\}_{k=1}^\infty$ such that
\begin{align}
\lim_{k\to\infty} d^{\mathbb{Z}_+}(X^{N_k},x) =0
\end{align}
where $x(0)=x^0$ and $x_{i,j}\in \mathcal{E}_c(x^0)$, for all $i$ and $j$.
\end{proposition}

Since Lipschitz continuity implies absolute continuity, we have obtained that limit points of $X^N(t)$ exist and are absolutely continuous.
Since all sample paths of $X^N(t)$ take values in $\mathcal{S}$, these limit points must belong as well to  $\mathcal{S}$ because  $\mathcal{S}$ is a closed set.
Therefore, to conclude the proof of Theorem~1 it remains to show that the derivative of $x_{i,j}(t)$ is as in Definition~1 for all $i$ and $j$,
provided that $t$ is a regular time.
This is done in the next subsection and
will also prove that a fluid solution started in $x^{(0)}\in\mathcal{S}_1$ exists.

\subsubsection{Limit trajectories are fluid solutions}
\label{ascascs}

Fix $\omega\in\mathcal{C}$ and let $\{X^{N_k}(\omega,t)\}_{k=1}^\infty$ be a subsequence that converges to $\overline{x}$ (by Proposition~\ref{PR:tightness}), i.e.
\begin{equation}
\label{seb_conv}
\lim_{k\to\infty} \sup_{t\in [0,T]} \| X^{N_k}(\omega, t) - \overline{x}(t)\|_w = 0.
\end{equation}
In the remainder, we fix such $\omega\in\mathcal{C}$ such that \eqref{seb_conv} holds and for simplicity we drop the dependency on $\omega$.
Since $\overline{x}$ must be Lipschitz continuous (by Proposition~\ref{PR:tightness}), it is also absolutely continuous and to conclude the proof of Theorem~1, it remains to show that $\overline{x}(t)$ satisfies the conditions on the derivatives given in Definition~1 whenever $\overline{x}_{i,j}(t)$ is differentiable, for all $i,j$.

We say that $t$ is a point of differentiability (of $\overline{x}$) if $x_{i,j}(t)$ is differentiable for all $i,j$.

We will (implicitly) use several times the following elementary lemma, which holds true because $\overline{x}$ is a non-negative absolutely continuous function.
\begin{lemma}
If $\overline{x}_{i,j}(t)=0$ and $t$ is a point of differentiability of $\overline{x}_{i,j}$, then $\dot{\overline{x}}_{i,j}(t)=0$.
\end{lemma}

Let $\epsilon>0$.
By Lemma~\ref{LM:tightness}, there exists a sequence $A_{N_k}\downarrow 0$ such that
$X_{i,j}^{N_k}(\omega,u) \in [ \overline{x}_{i,j}(t) -\epsilon \phi - A_{N_k}, \overline{x}_{i,j}(t) +\epsilon \phi + A_{N_k} ]$,
for all $u\in[t,t+\epsilon]$.
Thus, for all $k$ sufficiently large,
$X_{i,j}^{N_k}(\omega,u) \in [ \overline{x}_{i,j}(t) -2\epsilon \phi, \overline{x}_{i,j}(t) +2\epsilon \phi  ]$, for all $u\in[t,t+\epsilon]$.
Thus, we have
\begin{align}
\label{eq:2_L_eps}
|X_{i,j}^{N_k}(u)-\overline{x}_{i,j}(t)| \le 2 \phi \epsilon,\quad \forall u \in [t,t+\epsilon]
%
%
\end{align}
for all $k$ sufficiently large.
In addition, using \eqref{eq:2_L_eps} and that $g$ is Lipschitz, we obtain
\begin{subequations}
\label{eq:2_L_eps_g}
\begin{align}
|g(X^{N_k}(u)) - g(\overline{x}(u) )|
& \le
L \| X^{N_k}(u)) - \overline{x}(u) \|_w \\
%
&\le 2 \phi \epsilon L \sqrt{\sum_{i,j} \frac{1}{2^{i+j}}}
=   2 \phi \epsilon L \sqrt{2},
\end{align}
\end{subequations}
for all $u \in [t,t+\epsilon]$,
where $L$ is the Lipschitz constant of the scaling rule~$g$.

We will refer to the following lemma, which is a straightforward consequence of~\eqref{eq:2_L_eps} and of the strong law of the large numbers.
In points where the fluid drift function is continuous, it will provide an expression for terms related to job departures, server initializations/departures and, in some cases, dispatching decisions.
\begin{lemma}
\label{lemma_ascasji1}
Fix $\omega\in\mathcal{C}$ and let \eqref{seb_conv} hold.
Then,
\begin{footnotesize}
\begin{align*}
 \lim_{\epsilon\downarrow 0}\lim_{k\to\infty} \frac{1}{\epsilon N_k} \sum_{n=\mathcal{N}_\phi(N_k t)+1}^{\mathcal{N}_\phi(N_k (t+\epsilon))}
 \I{W_n=1}  \II{(Y_{i-1}^N(t_n^-),Y_i^N(t_n^-)]}{D_n}
 \,=  \overline{x}_{i,2}(t) \\
 \lim_{\epsilon\downarrow 0}\lim_{k\to\infty} \frac{1}{\epsilon N_k} \sum_{n=\mathcal{N}_\phi(N_k t)+1}^{\mathcal{N}_\phi(N_k (t+\epsilon))}
\I{W_n=3}   \II{(0,X_{0,1}^N(t_n^-)]}{I_n}
 \,=  \beta \overline{x}_{0,1}(t) \\
 \lim_{\epsilon\downarrow 0}\lim_{k\to\infty} \frac{1}{\epsilon N_k} \sum_{n=\mathcal{N}_\phi(N_k t)+1}^{\mathcal{N}_\phi(N_k (t+\epsilon))}
\I{W_n=4}
   \II{(0,X_{0,2}^N(t_n^-)]}{E_n}
 \,=  \gamma \overline{x}_{0,2}(t).
\end{align*}
\end{footnotesize}
In addition,
\begin{footnotesize}
\begin{multline*}
\lim_{\epsilon\downarrow 0}\lim_{k\to\infty} \frac{1}{\epsilon N_k} \sum_{n=\mathcal{N}_\phi(N_k t)+1}^{\mathcal{N}_\phi(N_k (t+\epsilon))}
 \I{W_n=0}
 \II{(Y_{i-1}^N(t_n^-),Y_{i}^N(t_n^-)]}{A_n^1 Y_d^N(t_n^-)}
\I{i\le d}
%
\\=
\frac{\lambda \I{i\le d} \overline{x}_i(t)}{\sum_{j=0}^d \overline{x}_{j,2}(t)}
%
\end{multline*}
\end{footnotesize}
provided that $\sum_{j=0}^d \overline{x}_{j,2}>0$, and
\begin{align*}
 \lim_{\epsilon\downarrow 0}\lim_{k\to\infty} \frac{1}{\epsilon N_k} \sum_{n=\mathcal{N}_\phi(N_k t)+1}^{\mathcal{N}_\phi(N_k (t+\epsilon))}
 \I{W_n=0} H_{i}(X^N(t_n^{-}))
 &\,=  \lambda h_i(\overline{x})
\end{align*}
provided that Power-of-$d$ is used.
\end{lemma}

\begin{proof}
Given in Section~\ref{sec:appendix_technical_lemmas}.
\end{proof}

The next proposition proves the desired condition on the amount of fluid of cold and initializing servers.
\begin{proposition}
\label{prop_asc9s}
Fix $\omega\in\mathcal{C}$, let \eqref{seb_conv} hold and assume that $t$ is a point of differentiability.
Then,
\begin{multline}
 \label{x00_a}
\dot{\overline{x}}_{0,0}  = \gamma \overline{x}_{0,2}(t) - \alpha \I{\overline{x}_{0,0}(t)>0}g(\overline{x}(t))
\\
-  \gamma \overline{x}_{0,2}(t)\,\I{\overline{x}_{0,0}(t)=0,\, \gamma \overline{x}_{0,2}(t) \le \alpha g(\overline{x}(t))}
\end{multline}
\begin{multline}
\label{x01_a}
\dot{\overline{x}}_{0,1}  = \alpha  g(\overline{x}(t)) \I{\overline{x}_{0,0}(t)>0} -\beta\overline{x}_{0,1}(t)
\\
+\gamma \overline{x}_{0,2}(t)\,\I{\overline{x}_{0,0}(t)=0,\, \gamma \overline{x}_{0,2}(t) \le \alpha g(\overline{x}(t))}.
\end{multline}
\end{proposition}
\begin{proof}
Assume that $\overline{x}_{0,0}(t)>0$ and let $\epsilon\in(0,\frac{\overline{x}_{0,0}(t)}{2 \phi})$.
Given that
\begin{align}
\label{t_nascsac}
t_n\in (t,t+\epsilon] \mbox{ if }  n \in \{ \mathcal{N}_\phi({N_k} t)+1, \ldots, \mathcal{N}_\phi({N_k} (t+\epsilon))\},
\end{align}
\eqref{eq:2_L_eps} implies that for all~$k$ sufficiently large,
$|X_{0,0}^{N_k}(t_n^-)-\overline{x}_{0,0}(t)| \le 2 \phi \epsilon < \overline{x}_{0,0}(t)$
and thus $X_{0,0}^{N_k}(t_n^-)  >0$.
We have shown that
\begin{equation}
\label{eq:krhfjc1}
\I{X_{0,0}^{N_k}(t_n^{-})>0} = 1, \quad \forall n \in \{ \mathcal{N}_\phi({N_k} t)+1, \ldots, \mathcal{N}_\phi({N_k} (t+\epsilon))\}
\end{equation}
for all $k$ sufficiently large.
Using \eqref{X_construction}, Lemma~\ref{lemma_ascasji1} and \eqref{eq:krhfjc1}, we have
\begin{subequations}
\label{asck7fgv_asc9s_PP}
\begin{align}
\nonumber
&\dot{\overline{x}}_{0,0}(t)
=\lim_{\epsilon\downarrow 0} \frac{1}{\epsilon} \lim_{k\to\infty} \left(X_{0,0}^{N_k}(t+\epsilon) - X_{0,0}^{N_k}(t)\right)\\
\nonumber
&=\lim_{\epsilon\downarrow 0} \lim_{k\to\infty} \frac{1}{\epsilon N_k}
\sum_{n=\mathcal{N}_\phi(N_k t)+1}^{\mathcal{N}_\phi(N_k (t+\epsilon))}
\bigg(
\I{W_n=4}
   \II{(0,X_{0,2}^{N_k}(t_n^-)]}{E_n}
\\
\label{asck7fgv_asc9s}
&\qquad\qquad - \I{W_n=2}
\I{X_{0,0}^{N_k}(t_n^{-})>0}
\, \II{(0,g(X^{N_k}(t_n^{-}))]}{R_n}\bigg)\\
\label{asck7fgv}
& = \gamma \overline{x}_{0,2}(t)
- \lim_{\epsilon\downarrow 0} \lim_{k\to\infty} \frac{1}{\epsilon N_k}
\sum_{{\scriptstyle n=\mathcal{N}_\phi(N_k t)+1:}\atop{\scriptstyle W_n=2}}^{\mathcal{N}_\phi(N_k (t+\epsilon))}
 \II{(0,g(X^{N_k}(t_n^{-}))]}{R_n}.
\end{align}
\end{subequations}
Since $t$ is a point of differentiability,
the double limit in the RHS of \eqref{asck7fgv} exists.
Then,~\eqref{eq:2_L_eps_g} implies that given~$\epsilon>0$ small enough,
$g(X^{N_k}(t_n^-)) \in [g(\overline{x}(t))-2 \phi \epsilon L \sqrt{2}, g(\overline{x}(t))+ 2 \phi \epsilon L \sqrt{2}]$
for all~$k$ sufficiently large.
Combining these bounds with Lemma~\ref{GC_F} and letting $\epsilon\downarrow 0$ (as in the proof of Lemma~\ref{lemma_ascasji1}), we obtain
\begin{align}
\label{lemma_g_formulazza}
\lim_{\epsilon\downarrow 0} \lim_{k\to\infty} \frac{1}{\epsilon N_k}
\sum_{{\scriptstyle n=\mathcal{N}_\phi(N_k t)+1:}\atop{\scriptstyle W_n=2}}^{\mathcal{N}_\phi(N_k (t+\epsilon))}
\, \II{(0,g(X^{N_k}(t_n^{-}))]}{R_n}
=\alpha g(\overline{x}(t)).
\end{align}
Similarly, on coordinates (0,1), we obtain
\begin{align}
\nonumber
&\dot{\overline{x}}_{0,1}(t)
=\lim_{\epsilon\downarrow 0} \frac{1}{\epsilon} \lim_{k\to\infty}
\left(
X_{0,1}^{N_k}(t+\epsilon) - X_{0,1}^{N_k}(t)
\right)\\
\nonumber
&=\lim_{\epsilon\downarrow 0} \lim_{k\to\infty}
\frac{1}{\epsilon N_k}
%
%
%
\sum_{{\scriptstyle n=\mathcal{N}_\phi(N_k t)+1:}\atop{\scriptstyle W_n=2}}^{\mathcal{N}_\phi(N_k (t+\epsilon))}
\I{X_{0,0}^{N_k}(t_n^{-})>0}
\, \II{(0,g(X^{N_k}(t_n^{-}))]}{R_n} \\
\nonumber
& \qquad-
\frac{1}{\epsilon N_k}
\sum_{{\scriptstyle n=\mathcal{N}_\phi(N_k t)+1:}\atop{\scriptstyle W_n=3}}^{\mathcal{N}_\phi(N_k (t+\epsilon))}
\I{W_n=3}
   \II{(0,X_{0,1}^{N_k}(t_n^-)]}{I_n}
\\
\nonumber
& = \alpha  g(\overline{x})
-\beta\overline{x}_{0,1}(t).
\end{align}

Now, let us assume that $\overline{x}_{0,0}(t)= 0$.
First, we notice that
\begin{align}
\label{asck7fgvascc}
&\dot{\overline{x}}_{0,0}(t)
 =
\gamma \overline{x}_{0,2}(t) -\lim_{\epsilon\downarrow 0} \lim_{k\to\infty}
\sum_{{\scriptstyle n=\mathcal{N}_\phi(N_k t)+1:}\atop{\scriptstyle W_n=2, X_{0,0}^N(t_n^{-})>0}}^{\mathcal{N}_\phi(N_k (t+\epsilon))}
\frac{ \II{(0,g(X^{N_k}(t_n^{-}))]}{R_n}}{\epsilon N_k}\\
\nonumber
& \ge \gamma \overline{x}_{0,2}(t) - \lim_{\epsilon\downarrow 0} \lim_{k\to\infty} \frac{1}{\epsilon N_k}
\sum_{{\scriptstyle n=\mathcal{N}_\phi(N_k t)+1:}\atop{\scriptstyle W_n=2}}^{\mathcal{N}_\phi(N_k (t+\epsilon))}
\II{(0,g(X^{N_k}(t_n^{-}))]}{R_n}
\\
&
= \gamma \overline{x}_{0,2}(t) - \alpha g(\overline{x})
\end{align}
where
%
the first equality follows by \eqref{asck7fgv_asc9s} and Lemma~\ref{lemma_ascasji1}, and
the last equality follows by~\eqref{lemma_g_formulazza}.
Thus, if $\overline{x}_{0,0}(t)=0$ and $\gamma \overline{x}_{0,2}(t) > \alpha g(\overline{x})$,
then by the previous inequality $\dot{\overline{x}}_{0,0}(t)>0$, which is not possible because
if $t$ is a point of differentiability and $\overline{x}_{0,0}(t)=0$ then necessarily $\dot{\overline{x}}_{0,0}(t)=0$ as $\overline{x}_{0,0}$ is a non-negative absolutely continuous function.
Thus, in a point of differentiability $t$ where $\overline{x}_{0,0}(t)=0$, we must have $\gamma \overline{x}_{0,2}(t) \le \alpha g(\overline{x})$.
and, necessarily, $\dot{\overline{x}}_{0,0}(t)=0$.
In this case, \eqref{asck7fgvascc} gives
\begin{multline}
\label{sc8as9ss}
\gamma \overline{x}_{0,2}(t) =
\lim_{k\to\infty} \frac{1}{\epsilon N_k}
\sum_{n=\mathcal{N}_\phi(N_k t)+1}^{\mathcal{N}_\phi(N_k (t+\epsilon))}\I{W_n=2}
\\ \times\I{X_{0,0}^N(t_n^{-})>0}
\, \II{(0,g(X^N(t_n^{-}))]}{R_n}.
\end{multline}
This term is interpreted as the amount of idle-on servers that become cold but instantly turn initializing.
Substituting~\eqref{sc8as9ss} in the previous equalities within the conditions $\gamma \overline{x}_{0,2}(t) \le \alpha g(\overline{x})$ and $\overline{x}_{0,0}(t)=0$, we obtain \eqref{x00_a} and \eqref{x01_a}.
\end{proof}

On the coordinates associated to warm servers, it remains to prove that
\begin{equation}
\label{scas8cas1}
\dot{\overline{x}}_{0,2}(t)
%
%
= \overline{x}_{1,2}(t) - \lambda h_0(\overline{x}(t)) + \beta \overline{x}_{0,1}(t) - \gamma \overline{x}_{0,2}(t)
\end{equation}
\begin{multline}
\label{scas8cas2}
\dot{\overline{x}}_{i,2}(t)
%
%
= \overline{x}_{i+1,2}(t)\I{i<B}-\overline{x}_{i,2}(t) \\
+ \lambda(h_{i-1} (\overline{x}(t))- h_i(\overline{x}(t))\I{i<B}),\, i\ge 1,
\end{multline}
whenever $t$ is a point of differentiability of $\overline{x}$. Let
\begin{align}
\label{eq:calH_i_def}
\mathcal{H}_i(t)\bydef \lim_{\epsilon\downarrow 0}\lim_{k\to\infty} \frac{1}{\epsilon N_k} \sum_{n=\mathcal{N}_\phi(N_k t) +1}^{\mathcal{N}_\phi(N_k (t+\epsilon)) }
\I{W_n=0}  H_{i}(t_n^{-})\ge 0,
\end{align}
which is interpreted as the rate at which jobs are assigned to warm servers with exactly $i$ jobs.
Using Lemma~\ref{lemma_ascasji1}  and \eqref{X_construction}, we have
\begin{subequations}
\label{scas8_HH}
\begin{align}
\dot{\overline{x}}_{0,2}(t)&
 =\lim_{\epsilon\downarrow 0} \frac{1}{\epsilon} \lim_{k\to\infty} \left(X_{0,2}^{N_k}(t+\epsilon) - X_{0,2}^{N_k}(t)\right)\\
& = \overline{x}_{1,2}(t) - \mathcal{H}_0(t) + \beta \overline{x}_{0,1}(t) - \gamma \overline{x}_{0,2}(t)\\
\dot{\overline{x}}_{i,2}(t)&
=\lim_{\epsilon\downarrow 0} \frac{1}{\epsilon} \lim_{k\to\infty}
\left( X_{i,2}^{N_k}(t+\epsilon) - X_{i,2}^{N_k}(t)\right)\\
&= \overline{x}_{i+1,2}(t)\I{i<B}-\overline{x}_{i,2}(t) + \mathcal{H}_{i-1}(t)- \mathcal{H}_{i}(t)\I{i<B}.
\end{align}
\end{subequations}

In the following, we need to show that $\mathcal{H}_i(t)=h_i(\overline{x}(t))$ where the $h_i$'s are as in Definition~1.
We treat the cases of Power-of-$d$ and JBT-$d$ separately.

\begin{lemma}
\label{lemma_Po_d}
Assume that Power-of-$d$ is applied.  Then,~\eqref{scas8cas1} and~\eqref{scas8cas2} hold true.
\end{lemma}

\begin{proof}
If $\overline{x}_{0,0}+\overline{x}_{0,1}<1$, then the structure of the $H_i$'s in \eqref{h_i_Pod_proof} and Lemma~\ref{lemma_ascasji1} immediately give \eqref{scas8cas1} and  \eqref{scas8cas2}.
Now, let us assume that $\overline{x}_{0,0}+\overline{x}_{0,1}=1$.
On coordinate (0,2), in a point of differentiability we necessarily have $\dot{\overline{x}}_{0,2}=0$.
Using Lemma~\ref{lemma_ascasji1}  and \eqref{X_construction}, we obtain
\begin{align}
\label{H0_pod_1}
&\dot{\overline{x}}_{0,2}(t)
=\lim_{\epsilon\downarrow 0} \frac{1}{\epsilon} \lim_{k\to\infty} \left( X_{0,2}^{N_k}(t+\epsilon) - X_{0,2}^{N_k}(t) \right)\\
& = \beta \overline{x}_{0,1}(t) - \mathcal{H}_0(t) = 0.
\end{align}
Similarly, on coordinate $(1,2)$, Lemma~\ref{lemma_ascasji1} and \eqref{H0_pod_1} imply that in a point of differentiability we have
$\dot{\overline{x}}_{1,2}(t)
= \mathcal{H}_0(t) - \mathcal{H}_1(t) = 0$
and thus $\mathcal{H}_1(t)=\mathcal{H}_0(t)=\beta \overline{x}_{0,1}(t)$.
Then, on coordinate $(i,2)$ by induction we obtain $\mathcal{H}_i(t)=\mathcal{H}_{i-1}(t)=\beta \overline{x}_{0,1}(t)$.
On the other hand, we also have
\begin{align*}
\dot{\overline{x}}_{0,2}(t)
%
%
& =\beta \overline{x}_{0,1}(t) -
\lim_{\epsilon\downarrow 0} \lim_{k\to\infty}
\frac{1}{\epsilon N_k}
\sum_{{\scriptstyle n=\mathcal{N}_\phi(N_k t)+1:}\atop{\scriptstyle W_n=0}}^{\mathcal{N}_\phi(N_k (t+\epsilon))}
H_{0}(t_n^{-})\\
& \ge \beta \overline{x}_{0,1}(t) -\lambda
\end{align*}
where in the last inequality we have just used that $ H_{0}(t_n^{-})\le 1$.
Thus, if $\beta \overline{x}_{0,1}(t) >\lambda$, we get a contradiction and $t$ can not be a point of differentiability.
Substituting $\mathcal{H}_i(t)=\beta \overline{x}_{0,1}(t)$ in \eqref{scas8_HH} when $\beta \overline{x}_{0,1}(t) \le \lambda$, we obtain~\eqref{scas8cas1}-\eqref{scas8cas2}.
\end{proof}

The case of JBT-$d$ is more delicate than Power-of-$d$ because of the discontinuous structure of the $H_i$'s when $\sum_{j=0}^d X_{j,2}^N(t_n^-)=0$, see~\eqref{h_i_JBTd_proof}.
In addition to a more involved argument than the one presented in the proof of Lemma~\ref{lemma_Po_d}, which we will develop in Lemma~\ref{lemma_JBT_d} below, we need the following lemma, which we will use to determine an expression for $\mathcal{H}_i$ when $\sum_{j=0}^d \overline{x}_{j,2}(t)=0$.
\begin{lemma}
\label{lemma_ascasji2}
Assume that $\overline{x}(t)$ satisfies $\overline{x}_{0,0}(t)+\overline{x}_{0,1}(t)<1$.
Then, \eqref{asjc78dcckiujff} holds true for all $i$.
\begin{figure*}[!t]
\normalsize
%
%
%
\begin{multline}
\label{asjc78dcckiujff}
\lim_{\epsilon\downarrow 0}\lim_{k\to\infty} \frac{1}{\epsilon N_k} \sum_{n=\mathcal{N}_\phi(N_k t) +1}^{\mathcal{N}_\phi(N_k (t+\epsilon)) }
\I{W_n=0}  \II{(Y_{i-1}^{N_k}(t_n^-),Y_{i}^{N_k}(t_n^-)]}{A_n^1 (1-X_{0,0}^{N_k}(t_n^-)-X_{0,1}^{N_k}(t_n^-) )}
\I{\sum_{j=0}^d X_{j,2}^{N_k}(t_n^-)>0}\\
=
\frac{\overline{x}_{i,2}(t)}{1-\overline{x}_{0,0}(t)-\overline{x}_{0,1}(t)}
\lim_{\epsilon\downarrow 0}\lim_{k\to\infty} \frac{1}{\epsilon N_k} \sum_{n=\mathcal{N}_\phi(N_k t) +1}^{\mathcal{N}_\phi(N_k (t+\epsilon)) }
\I{W_n=0}
\I{\sum_{j=0}^d X_{j,2}^{N_k}(t_n^-)>0}.
\end{multline}
%
%
%
%
\hrulefill
\vspace*{4pt}
\end{figure*}
\end{lemma}
\begin{proof}
Given in Section~\ref{sec:appendix_technical_lemmas}.
\end{proof}

The following lemma proves the desired property in the case of JBT-$d$.

\begin{lemma}
\label{lemma_JBT_d}
Assume that JBT-$d$ is applied. Then,~\eqref{scas8cas1} and~\eqref{scas8cas2} hold true.
\end{lemma}
\begin{proof}
We analyze $\mathcal{H}_i$ and the resulting expression will be substituted in \eqref{scas8_HH}.
This will give~\eqref{scas8cas1} and~\eqref{scas8cas2}.

First, if $\overline{x}_{0,0}+\overline{x}_{0,1}=1$, the argument in the proof of Lemma~\ref{lemma_Po_d} gives i) $\mathcal{H}_i(t)=\beta \overline{x}_{0,1}(t)$ when $\beta \overline{x}_{0,1}(t) \le \lambda$
and ii) $t$ not a point of differentiability when $\beta \overline{x}_{0,1}(t) > \lambda$.
This gives ~\eqref{scas8cas1} and~\eqref{scas8cas2}  (when $\overline{x}_{0,0}+\overline{x}_{0,1}=1$) and in the remainder we assume that $\overline{x}_{0,0}+\overline{x}_{0,1}<1$.

Let us now assume that  $\sum_{j=0}^d \overline{x}_{j,2}(t)>0$
and let $\epsilon\in(0,\frac{\sum_{j=0}^d \overline{x}_{j,2}(t)}{2 \phi (d+1)})$.
Since $t_n\in (t,t+\epsilon]$ whenever $n \in \{ \mathcal{N}_\phi({N_k} t)+1, \ldots, \mathcal{N}_\phi({N_k} (t+\epsilon))\}$,
\eqref{eq:2_L_eps} and the triangular inequality imply that for all~$k$ sufficiently large
$|\sum_{j=0}^d X_{j,2}^{N_k}(t_n) - \overline{x}_{j,2}(t)| \le 2(d+1) \phi \epsilon < \sum_{j=0}^d \overline{x}_{j,2}(t)$
and thus $\sum_{j=0}^d X_{j,2}^{N_k}(t_n)  >0$.
We have shown that
\begin{footnotesize}
\begin{equation}
\label{eq:krhfjc12}
\I{\sum_{j=0}^d X_{j,2}^{N_k}(t_n^{-})>0} = 1, \forall n \in \{ \mathcal{N}_\phi({N_k} t)+1, \ldots, \mathcal{N}_\phi({N_k} (t+\epsilon))\}
\end{equation}
\end{footnotesize}
for all $k$ sufficiently large, given $\epsilon>0$ sufficiently small.
Substituting \eqref{eq:krhfjc12} in \eqref{h_i_JBTd_proof}
and applying Lemma~\ref{lemma_ascasji1}, we obtain~\eqref{scas8cas1} and~\eqref{scas8cas2} (under the conditions $\overline{x}_{0,0}+\overline{x}_{0,1}<1$ and  $\sum_{j=0}^d \overline{x}_{j,2}(t)>0$).

It remains to understand the terms $\mathcal{H}_i$ in the case where $\sum_{j=0}^d \overline{x}_{j,2}(t)=0$, which we assume in the remainder of the proof.

Suppose that~$t$ is a point of differentiability.
Then,
by applying Lemma~\ref{lemma_ascasji1} to~$X_{0,2}^N$ (see~\eqref{X_construction}), we obtain
\begin{align}
\nonumber
&\dot{\overline{x}}_{0,2}(t)
=\lim_{\epsilon\downarrow 0} \frac{1}{\epsilon} \lim_{k\to\infty} X_{0,2}^{N_k}(t+\epsilon) - X_{0,2}^{N_k}(t)
\\
\label{H0_JBTd_1}
& = \overline{x}_{1,2}(t)\I{d=0} + \beta \overline{x}_{0,1}(t) - \mathcal{H}_0(t),
\end{align}
and given that necessarily $\dot{\overline{x}}_{0,2}(t)=0$, we obtain
\begin{align}
\label{H0_JBTd_2}
\mathcal{H}_0(t) = \overline{x}_{1,2}(t)\I{d=0} + \beta \overline{x}_{0,1}(t).
\end{align}
Similarly, on coordinate $(i,2)$, with $0<i\le d$, we obtain
\begin{align}
\label{Hi_JBTd_1}
&\dot{\overline{x}}_{i,2}(t)
= \overline{x}_{i+1,2}(t)\I{i=d} + \mathcal{H}_{i-1}(t)- \mathcal{H}_i(t) =0.
\end{align}
By induction, this gives
$\mathcal{H}_i(t)=\mathcal{H}_0(t)=\beta \overline{x}_{0,1}(t)$ for all $i<d$ and $\mathcal{H}_d(t) = \beta \overline{x}_{0,1}(t)+\overline{x}_{d+1,2}(t)$, that is,
\begin{align}
\label{eq:Hi_sc9asc}
\mathcal{H}_i(t) =\beta \overline{x}_{0,1}(t)+\overline{x}_{d+1,2}(t)\I{i=d},\qquad i\le d.
\end{align}
We have proven \eqref{eq:Hi_sc9asc} under the hypothesis that $t$ was a point of differentiability but now we show that $\overline{x}(t)$ is not differentiable if $\lambda< \overline{x}_{d+1,2} + (d+1)\beta \overline{x}_{0,1}$.
Towards this purpose, first we notice that
\begin{align*}
&\sum_{i=0}^d \mathcal{H}_i(t)
 = \lim_{\epsilon\downarrow 0}\lim_{k\to\infty} \frac{1}{\epsilon N_k} \sum_{n=\mathcal{N}_\phi(N_k t) +1}^{\mathcal{N}_\phi(N_k (t+\epsilon)) }
\I{W_n=0}  \sum_{i=0}^d  H_{i}(t_n^{-})\\
& = \lim_{\epsilon\downarrow 0}\lim_{k\to\infty} \frac{1}{\epsilon N_k} \sum_{n=\mathcal{N}_\phi(N_k t) +1}^{\mathcal{N}_\phi(N_k (t+\epsilon)) }
\I{W_n=0}  \I{\sum_{j=0}^d X_{j,2}^N(t_n^-)>0}
\\
%
%
& \le \lim_{\epsilon\downarrow 0}\lim_{k\to\infty} \frac{1}{\epsilon N_k} \sum_{n=\mathcal{N}_\phi(N_k t) +1}^{\mathcal{N}_\phi(N_k (t+\epsilon)) }
\I{W_n=0}  =\lambda.
\end{align*}
Here, the first equality follows because the limits $\mathcal{H}_i(t)$ exist
and the second inequality follows by the fact that (recall the definition of $\mathcal{H}_i$ in \eqref{eq:calH_i_def})
\begin{multline}
\label{qwdqwdwdc}
\sum_{i=0}^d H_i (t_n^{-})
=
 \I{\sum_{j=0}^d X_{j,2}^N(t_n^-)>0}\\
 +
\II{(0,Y_{d}^N(t_n^-)]}{A_n^1 (1-X_{0,0}^N(t_n^-)-X_{0,1}^N(t_n^-) )}
\I{\sum_{j=0}^d X_{j,2}^N(t_n^-)=0}
\end{multline}
and by Lemma~\ref{lemma_ascasji1} because $\sum_{j=0}^d \overline{x}_{j,2}(t)=0$.
Then, using \eqref{eq:Hi_sc9asc}, we necessarily have
\begin{align}
\label{ascu8ascs}
\sum_{i=0}^d \mathcal{H}_i(t)
= \overline{x}_{d+1,2}(t) + (d+1) \beta \overline{x}_{0,1}(t) \le \lambda
\end{align}
and, given that necessarily $\mathcal{H}_i\ge 0$, we conclude that~$t$ can not be a point of differentiability whenever \eqref{ascu8ascs} does not hold true.

Now, we investigate $\mathcal{H}_i$ when $i>d$ and assuming that \eqref{ascu8ascs} holds as otherwise $\overline{x}(t)$ would not be differentiable.
We observe that
\begin{align*}
&\mathcal{H}_i(t)
=
\frac{\lambda \overline{x}_{i,2}(t)}{1-\overline{x}_{0,0}(t)-\overline{x}_{0,1}(t)} \\
& - \lim_{\epsilon\downarrow 0}\lim_{k\to\infty} \frac{1}{\epsilon N_k}
\sum_{{\scriptstyle n=\mathcal{N}_\phi(N_k t)+1:}\atop{\scriptstyle W_n=0}}^{\mathcal{N}_\phi(N_k (t+\epsilon))}
\II{(Y_{i-1}^N(t_n^-),Y_{i}^N(t_n^-)]}{A_n^1 (1-X_{0,0}^N(t_n^-)-X_{0,1}^N(t_n^-) )}\\
& \qquad \times \I{\sum_{j=0}^d X_{j,2}^N(t_n^-)>0}\\
&=
\frac{\lambda \overline{x}_{i,2}(t)}{1-\overline{x}_{0,0}(t)-\overline{x}_{0,1}(t)}
-
\frac{\overline{x}_{i,2}(t)}{1-\overline{x}_{0,0}(t)-\overline{x}_{0,1}(t)}
\\
&\quad \times \lim_{\epsilon\downarrow 0}\lim_{k\to\infty} \frac{1}{\epsilon N_k}
\sum_{{\scriptstyle n=\mathcal{N}_\phi(N_k t)+1:}\atop{\scriptstyle W_n=0}}^{\mathcal{N}_\phi(N_k (t+\epsilon))}
\I{\sum_{j=0}^d X_{j,2}^N(t_n^-)>0} \\
&=
\frac{\lambda \overline{x}_{i,2}(t)}{1-\overline{x}_{0,0}(t)-\overline{x}_{0,1}(t)}
-
\frac{\overline{x}_{i,2}(t)}{1-\overline{x}_{0,0}(t)-\overline{x}_{0,1}(t)}
\sum_{i=0}^d \mathcal{H}_i(t)\\
&=
\overline{x}_{i,2}(t) \frac{\lambda - \overline{x}_{d+1,2}(t) - (d+1) \beta \overline{x}_{0,1}(t) }{1-\overline{x}_{0,0}(t)-\overline{x}_{0,1}(t)}.
\end{align*}
In the first equality, we have used \eqref{h_i_JBTd_proof} and applied Lemma~\ref{lemma_ascasji1} to the definition of $\mathcal{H}_i$ in \eqref{eq:calH_i_def};
In the second, we have applied Lemma~\ref{lemma_ascasji2}.
In the third,
we have used \eqref{qwdqwdwdc}
and that
\begin{align*}
 0 &\le \lim_{\epsilon\downarrow 0}\lim_{k\to\infty} \frac{1}{\epsilon N_k}
\II{(0,Y_{d}^N(t_n^-)]}{A_n^1 (1-X_{0,0}^N(t_n^-)-X_{0,1}^N(t_n^-) )}
\I{\sum_{j=0}^d X_{j,2}^N(t_n^-)=0}\\
&   \le \lim_{\epsilon\downarrow 0}\lim_{k\to\infty} \frac{1}{\epsilon N_k}
\II{(0,Y_{d}^N(t_n^-)]}{A_n^1 (1-X_{0,0}^N(t_n^-)-X_{0,1}^N(t_n^-) )}=0
\end{align*}
with the last inequality following by Lemma~\ref{lemma_ascasji1} as $\sum_{j=0}^d \overline{x}_{j,2}(t)=0$;
in the fourth, we have substituted~\eqref{eq:Hi_sc9asc}.
This concludes the proof.
\end{proof}

Thus, we have shown that $\overline{x}$ is a fluid solution.

\subsection{Proof of Theorem~2: fixed points}
\label{sec:proof_FP}

We now prove Theorem~2.
By definition, $x\in\mathcal{S}_1$  is a fixed point if and only if
\begin{subequations}
\label{x_FP}
\begin{align}
\label{x00_FP}
0 & = \gamma {x}_{0,2} - \alpha g\I{{x}_{0,0}>0}
-  \gamma {x}_{0,2}\,\I{{x}_{0,0}=0,\, \gamma {x}_{0,2} \le \alpha g}\\
\label{x01_FP}
0 & = \alpha  g \I{{x}_{0,0}>0} -\beta{x}_{0,1}
+\gamma {x}_{0,2}\,\I{{x}_{0,0}=0,\, \gamma {x}_{0,2} \le \alpha g} \\
\label{x02_FP}
0 &= x_{1,2} -  h_0(x) + \beta x_{0,1} - \gamma x_{0,2} \\
\label{xi2_FP}
0 &= x_{i+1,2} - x_{i,2} +  h_{i-1} (x)- h_i(x), \quad i\ge 1.
\end{align}
\end{subequations}
Together with $\|x\|=1$, we now show that these conditions coincide with (7)-(9).

If i) $x_{0,0}=0$ and $\gamma {x}_{0,2} > \alpha g$,
or
if ii) $x_{0,0}+x_{0,1}=1$,
then we easily observe that $x$ cannot be a fixed point.
Therefore, in the following we exclude these conditions.
Now,
summing \eqref{x00_FP} and \eqref{x01_FP}, we obtain
\begin{equation}
\label{ajvdfu8bf}
\beta x_{0,1} = \gamma x_{0,2}
\end{equation}
which gives (7b).
Then,  (7c) and  (7d) directly follow from \eqref{x00_FP} and  \eqref{x01_FP}.

Substituting \eqref{ajvdfu8bf} in \eqref{x02_FP}, the conditions \eqref{x02_FP}-\eqref{xi2_FP} become
\begin{subequations}
\label{standard_FP}
\begin{align}
0 &= x_{1,2} - h_0(x)\\
\label{standard_FP2}
0 &= x_{i+1,2}  -x_{i,2} +  h_{i-1} (x)- h_i(x) , \quad i\ge 1,
\end{align}
\end{subequations}
and taking summations
\begin{align}
\label{sj8x}
x_{i,2} = h_{i-1}(x), \quad i\ge 1.
\end{align}
The equations in~\eqref{standard_FP} are interpreted as the mean-field fixed-point equations associated to Power-of-$d$ and JBT-$d$ when the number of servers is $N y_0$ instead of $N$; we recall that $y_i=\sum_{i\ge 0} x_{i,2}$ is the proportion of warm servers with at least~$i$ jobs.
%
Within Power-of-$d$, one can directly check that for any given $x_{0,2}$,~\eqref{sj8x} holds if and only if $x_{i,2}$
is given by (8) and that, after a substitution, this gives~$\sum_{i\ge 1 } x_{i,2}=\lambda$ so that~(7a) must hold true.
The following lemma,
given in Section~\ref{sec:appendix_technical_lemmas},
handles the more delicate case of JBT-$d$.

\begin{lemma}
\label{asjch7rrr}
Within JBT-$d$, for any given $x_{0,2}$, \eqref{standard_FP} holds if and only if $x_{i,2}$ satisfies (9a)-(9e).
In addition, (7a) holds true.
\end{lemma}

Therefore, the conditions in \eqref{x_FP} are equivalent to~(7)-(9).
This proves the first statement of Theorem~2.
%

Now, under Assumption~3, $x_{i,2}$ is a function of $x_{0,2}$, for all $i\ge 1$, and we write $x_{i,2}$ as a shorthand notation for $x_{i,2}(x_{0,2})$.
Using \eqref{ajvdfu8bf},
we can then focus only on the following conditions:
\begin{subequations}
\label{suc7dg}
\begin{align}
\label{suc7dg_a}
x_{0,0} + \left(\frac{\gamma}{\beta} + 1\right)x_{0,2} &= 1  -  \lambda \\
\gamma x_{0,2}& \le \alpha g,\quad \mbox{ if } x_{0,0}=0 \\
\gamma x_{0,2}& = \alpha g,\quad  \mbox{ if } x_{0,0}>0.
\end{align}
\end{subequations}
Here, we notice that
$(x_{0,0}^\circ,x_{0,2}^\circ) = \big(0, \frac{\beta}{\beta+\gamma}(1-\lambda)\big)$
uniquely solves~\eqref{suc7dg}
if
\begin{align}
\label{casei}
%
\left(\frac{1}{\beta}+\frac{1}{\gamma}\right) \alpha g(x^\circ) \ge  1  -  \lambda
\end{align}
where $x^\circ$ is uniquely determined by $(x_{0,0}^\circ,x_{0,2}^\circ)$. 
So, let us assume that \eqref{casei} does \emph{not} hold true.
Then, if a point $(x_{0,0},x_{0,2})\in[0,1)^2$ that solves \eqref{suc7dg} exists,
then necessarily $x_{0,0}>0$ as otherwise $x_{0,2}=x_{0,2}^\circ$ (by \eqref{suc7dg_a}) and \eqref{casei} would hold, contradicting the hypothesis.
This proves the second part of Theorem~2.


\subsection{Proof of Theorem~3: fluid optimality}

\label{sec:proof_AO}


The non-linear structure taken by the $h_i$'s when $x_{0,2}=0$, see~(6), complicates the analysis and the identification of a Lyapunov function.
%
For this reason, our strategy is based on a divide-and-conquer approach.
This will actually provide insights about the dynamics followed by fluid solutions.
For simplicity, we provide a proof assuming that~$B<\infty$, which is essentially equivalent to assume that $x_{i,2}(0)=0$ for all $i$ large enough; this is not critical as $x_{i,2}^\star=0$ for all $i\ge 2$.
%
%

Let
$\overline{Q}(x)\bydef \sum_{i= 1}^Bix_{i,2}$,
i.e.,
the overall number of jobs in the system in state~$x$.
The following lemma
gives a property on the time derivative of $\overline{Q}(x(t))$.

\begin{lemma}
\label{lemma:Qdot}
Let $x(t)$ be a fluid solution induced by JIQ such that $x(0)\in\mathcal{S}_1$ and $B<\infty$. If~$t$ is a point of differentiability, then
\begin{equation}
\dot{\overline{Q}}(x(t)) = \lambda - y_1(t).
\end{equation}
\end{lemma}
\begin{proof}
First, we notice that
\begin{align}
\label{asckj8ff}
\dot{\overline{Q}}(x(t))
=  \sum_{i\ge 1} i \dot{x}_{i,2} (t)
 = - y_1 + \sum_{i=0}^{B-1} h_i(x(t))
%
\end{align}
where the second equality follows by applying Definition~1.
Now, we treat the cases $x_{0,2}(t)>0$ and $x_{0,2}(t)=0$ separately.
Suppose that $x_{0,2}(t)>0$. Then, $h_i(x(t))=\lambda\I{i=0}$ (by (6)) and substituting in \eqref{asckj8ff} we immediately get
$ \dot{\overline{Q}}(x(t))
%
%
%
%
=  \lambda - y_{1}(t) $ as desired.
Thus, suppose in the remainder that $x_{0,2}(t)=0$.
Now, assume that $y_0(t)>0$. Then, using again (6),
\begin{align*}
\dot{\overline{Q}}(x(t))
%
%
 = &- y_1 + \left(\beta {x}_{0,1}+{x}_{1,2} \right)\I{{x}_{1,2} + \beta {x}_{0,1}\le \lambda}
\\ & +  \frac{ y_1}{y_0}(\lambda - {x}_{1,2} - \beta {x}_{0,1}  )^+ \\
= & - y_1 + \left(\beta {x}_{0,1}+{x}_{1,2} \right)\I{{x}_{1,2} + \beta {x}_{0,1}\le \lambda} \\
&  +  (\lambda - {x}_{1,2} - \beta {x}_{0,1}  )^+
\end{align*}
and the statement follows immediately if ${x}_{1,2}(t) + \beta {x}_{0,1}(t)\le \lambda$.
On the other hand, if ${x}_{1,2}(t) + \beta {x}_{0,1}(t)> \lambda$,
then, since $x_{0,2}(t)=0$ and $t$ is supposed to be a point of differentiability, we get (by~(4c)) the contradiction that
$0 = \dot{x}_{0,2}(t)
= x_{1,2}(t) - h_0(x(t)) + \beta x_{0,1}(t) - \gamma x_{0,2}(t)
= x_{1,2}(t)  + \beta x_{0,1}(t) >\lambda
$; the first equality holds because $x_{0,2}(t) $ is a non-negative absolutely continuous function.
This shows that $t$ cannot be a point of differentiability.
Finally, if $y_0(t)=0$, then the differentiability at $t$ and the normalizing condition $\|x\|=1$ give $0 = \dot{y}_0(t) = -\dot{x}_{0,0}(t)-\dot{x}_{0.1}(t) = \beta x_{0,1}$ and thus $x_{0,0}(t)=1$.
Assumption~2 requires that $g(x)>0$ when $x_{0,0}=1$, so
(4a) implies that $\dot{{x}}_{0,0} <0$. This contradicts that $t$ is a point of differentiability because $x_{0,0}(t)$ is uniformly bounded by one and absolutely continuous.
\end{proof}

We now prove Theorem~3 by showing that $\|x(t)-x^\star \|\to0$ in each of the following complete and mutually exclusive cases.
For each case, we show that $x(t)$ follows a unique trajectory that stays in $\mathcal{S}_1$.

\ \\
\emph{Case i)}.
Suppose that $x_{0,2}(t)=0 $ for all $t\ge0$.
This rules out the possibility that $x_{0,0}(t)$  stays on zero for all $t$ large enough because (4a) and (4b), together with the normalizing condition $\|x\|=1$, would imply that $y_1(t)\to 1$ as $t\to\infty$, and in this case Lemma~\ref{lemma:Qdot} yields the contradiction that $\overline{Q}(x(t))$ is eventually negative.
Thus, without loss of generality, let us assume that $x_{0,0}(0)>0$.
Then, using~(4), $x(t)$ satisfies
\begin{subequations}
\label{gx_proof}
\begin{align}
\label{gx_1}
\dot{{x}}_{0,0} & = - \alpha g(x)\\
\label{gx_2}
\dot{{x}}_{0,1} & = \alpha  g(x)  -\beta{x}_{0,1}\\
\label{gx_3}
\dot{x}_{0,2} &= 0,\quad x_{1,2} + \beta x_{0,1} \le \lambda.
%
%
\end{align}
\end{subequations}
Note that $\lim_{t\to\infty} x_{0,0}(t)$ exists, say $x_{0,0}(\infty)$, because $\dot{x}_{0,0}(t)\le 0$ and  $x_{0,0}(t)$ is uniformly bounded.
Thus, as $t\to\infty$,
$\dot{x}_{0,0}(t)
=- \alpha g(x(t))
\to 0$.
Given the assumptions on $g$, $(\lambda - x_{0,1}(t)-\beta x_{1,2}(t))^+\to 0$ and since $x_{1,2}(t) + \beta x_{0,1}(t) \le \lambda$ for all $t$, by \eqref{gx_3},
we obtain that $x_{1,2}(t)+\beta x_{0,1}(t) \to \lambda$.
Then, \eqref{gx_2} and $g(x(t))\to 0$ imply that $x_{0,1}(t)\to 0 $ and thus $x_{1,2}(t)\to \lambda$.
In turn, (4d) gives $x_{i,2}(t)\to 0 $ for all $i\ge 2$, and the normalizing condition $\|x\|=1$ implies that necessarily $x_{0,0}(t)\to 1-\lambda$. Thus, $\|x(t)-x^\star\|\to 0$.

\ \\
\emph{Case ii)}.
Suppose that $x_{0,2}(t)>0$ for all $t$. Then, $x(t)$ satisfies the following conditions (using Definition~1)
\begin{subequations}
\label{Xvba}
\begin{align}
\label{Xvba1}
\dot{{x}}_{0,0} & = \gamma {x}_{0,2} - \alpha g\I{{x}_{0,0}>0}
-  \gamma {x}_{0,2}\,\I{{x}_{0,0}=0,\, \gamma {x}_{0,2} \le \alpha g}\\
\label{Xvba2}
\dot{{x}}_{0,1} & = \alpha  g \I{{x}_{0,0}>0} -\beta{x}_{0,1}
+\gamma {x}_{0,2}\,\I{{x}_{0,0}=0,\, \gamma {x}_{0,2} \le \alpha g} \\
\label{Xvba3}
\dot{x}_{0,2} &= x_{1,2} - \lambda + \beta x_{0,1} - \gamma x_{0,2} \\
\label{Xvba4}
\dot{x}_{1,2} &= x_{2,2}-x_{1,2} + \lambda \\
\label{Xvba5}
\dot{x}_{i,2} &= x_{i+1,2}\I{i< B}-x_{i,2} ,\quad i\ge2.
\end{align}
\end{subequations}
%
%
The ODE system~\eqref{Xvba4}-\eqref{Xvba5} is an autonomous linear ODE system with constant coefficients and, developing the matrix-exponential general solution of such ODE system, for all $i\ge 1$ we obtain
\begin{align}
\label{asic9as}
x_{i,2}(t) = \lambda \I{i=1} + e^{-t} \sum_{k=i}^B \frac{t^{k-i}}{(k-i)!} (x_{k,2}(0) - \lambda \I{k=1})
\end{align}
and thus $x_{i,2}(t) \to \lambda \I{i=1}$ as $t\to\infty$.
In turn,
$\lim_{t\to\infty} (\lambda-x_{1,2}(t)-\beta x_{0,1}(t))^+ = \lim_{t\to\infty} (-\beta x_{0,1}(t))^+ = 0$
and therefore
$g(x(t))
\to 0$.
Since $g(x(t))\to 0$, $x_{0,1}(t)\to 0$ necessarily by \eqref{Xvba2}, and using this in \eqref{Xvba3} we obtain $x_{0,2}(t)\to 0$.
Since $\|x\|=1$, necessarily $x_{0,0}(t)\to 1 - \lambda$ and we have shown that $\|x(t)-x^\star\|\to 0$.

\ \\
\emph{Case iii)}.
If the conditions in cases \emph{i)} and \emph{ii)} are not met, then there exists
$t_0$, $t_1$, with $t_0\le t_1<\infty$, and $\delta>0$
such that
\begin{enumerate}
 \item
$x_{0,2}(t)=0$ for all $t\in[t_0,t_1]$
 \item
$x_{0,2}(t)>0$ and $\dot{x}_{0,2}(t)<0$ for all $t\in[t_0-\delta,t_0)$, and
 \item
$x_{0,2}(t)>0$ and $\dot{x}_{0,2}(t)>0$ for all $t\in(t_1,t_1+\delta]$.
\end{enumerate}
On $[t_0-\delta,t_0)$, $h_0(x(t))=\lambda$ (by~(6)) and using~(4c), we obtain
$\dot{x}_{0,2} (t) = x_{1,2}(t) - \lambda + \beta x_{0,1}(t) - \gamma x_{0,2}(t) <0$
and thus by continuity
\begin{align}
\nonumber
0\ge& \lim_{t\uparrow t_0} x_{1,2}(t) - \lambda + \beta x_{0,1}(t) - \gamma x_{0,2}(t)\\
\label{asc8ddssz}
=& x_{1,2}(t_0) - \lambda + \beta x_{0,1}(t_0).
\end{align}
Since $x_{0,2}(t_0)=0$ on $[t_0,t_1]$, (4c) implies that \eqref{asc8ddssz} holds as well on $[t_0,t_1]$.
On $(t_1,t_1+\delta]$, $h_0(x(t))=\lambda$ (by~(6)) and using again~(4c), we obtain
\begin{align*}
0 &< \dot{x}_{0,2} (t)
= x_{1,2}(t) - h_0(x(t)) + \beta x_{0,1}(t) - \gamma x_{0,2}(t)\\
%
&< x_{1,2}(t) - \lambda + \beta x_{0,1}(t)
\end{align*}
and therefore
$g(x(t))
=0$.
By continuity of fluid solutions, $x_{1,2}(t_1)+\beta x_{0,1}(t_1)=\lambda$.
In addition,
on $(t_1,t_1+\delta]$, $x(t)$ is uniquely defined by
\begin{subequations}
\label{Xvbav}
\begin{align}
\label{Xvba1v}
\dot{{x}}_{0,0} & = \gamma {x}_{0,2}\\
\label{Xvba2v}
\dot{{x}}_{0,1} & = -\beta{x}_{0,1} \\
\label{Xvba3v}
\dot{x}_{0,2} &= x_{1,2} - \lambda + \beta x_{0,1} - \gamma x_{0,2} \\
\label{Xvba4v}
\dot{x}_{1,2} &= x_{2,2}-x_{1,2} + \lambda \\
\label{Xvba5v}
\dot{x}_{i,2} &= x_{i+1,2}\I{i< B}-x_{i,2} ,\quad i\ge2,
\end{align}
\end{subequations}
and we also know that $\dot{x}_{0,2}(t)>0$.
As long as \emph{a)} $g(x(t))=0$
and \emph{b)} $x_{0,2}(t)>0$, on $[t_1,\infty)$ the fluid solution under investigation $x(t)$ is indeed uniquely given by the trajectory induced by \eqref{Xvbav} on $[t_1,\infty)$.
In the remainder, we show that both \emph{a)} and \emph{b)} hold true for all $t$.
This will conclude the proof because $x^\star$ is the unique fixed point of \eqref{Xvbav} and because \eqref{Xvbav} is a linear ODE system with constant coefficients.
For simplicity of notation, let us shift time and assume that $t_1=0$.
Now,
since $x_{0,1}(t)  = x_{0,1}(0)e^{-\beta t} $ (by~\eqref{Xvba2v})
and
since $x_{1,2}(t)$ takes the form given in~\eqref{asic9as}, substituting in
\eqref{Xvba3v} we obtain
\begin{align*}
\dot{x}_{0,2}(t)
%
&=
\beta x_{0,1}(0)e^{-\beta t} - \gamma x_{0,2}(t)
+ e^{-t} (x_{1,2}(0) - \lambda ) \\
& \qquad + e^{-t} \sum_{k=2}^B \frac{t^{k-1}}{(k-1)!} x_{k,2}(0)
\\
&=
\beta x_{0,1}(0)e^{-\beta t} - \gamma x_{0,2}(t)
-\beta x_{0,1}(0) e^{-t} \\
& \qquad + e^{-t} \sum_{k=2}^B \frac{t^{k-1}}{(k-1)!} x_{k,2}(0)
\\
&\ge
\beta x_{0,1}(0)(e^{-\beta t}-e^{-t}) - \gamma x_{0,2}(t).
\end{align*}
Thus,
$x_{0,2}(t)\ge z(t) $ where $z(t)$ is uniquely defined by
$\dot{z}(t)  = \beta x_{0,1}(0)(e^{-\beta t}-e^{-t}) - \gamma z(t)$
with $z(0)= x_{0,2}(0)$.
The solution of this differential equation is
\begin{align*}
z(t)
=
\beta x_{0,1}(0) e^{-\gamma t}
\left(
\frac{ 1 - e^{-t\left(\beta-\gamma\right)}}{\beta-\gamma}
-
\frac{1 - e^{-t\left(1-\gamma\right)}}{1-\gamma}
\right)
\end{align*}
and now we notice that $z(t)>0$ if $\beta>1$,  for all $t$.
This proves property \emph{b)}.
To prove property \emph{a)}, we use again \eqref{asic9as} and $x_{1,2}(t_1)+\beta x_{0,1}(t_1)=\lambda$ to obtain
\begin{multline}
\label{ldokivuh85}
x_{1,2}(t) + \beta x_{0,1}(t) - \lambda
=
\beta x_{0,1}(0)\left( e^{-\beta t}
-e^{-t} \right)\\
  + e^{-t} \sum_{k=2}^B \frac{t^{k-1}}{(k-1)!}   x_{k,2}(0)
%
%
>0,
\end{multline}
where the last inequality follows because $\beta<1$.
Given~(18), \eqref{ldokivuh85} implies $g(x(t))=0$.

\section{Proofs of technical lemmas}
\label{sec:appendix_technical_lemmas}

\subsection{Proof of Lemma~\ref{lemma_ascasji1}}

We give a proof for the first limit because the argument used for the others is identical.


Since $t_n\in (t,t+\epsilon]$ whenever $n \in \{ \mathcal{N}_\phi({N_k} t)+1, \ldots, \mathcal{N}_\phi({N_k} (t+\epsilon))\}$,
\eqref{eq:2_L_eps} implies that for all~$k$ sufficiently large
$|Y_{i}^{N_k}(t_n)-\sum_{j=0}^i \overline{x}_{j,2}(t)| \le C \epsilon$, for some constant $C$, i.e.,
\begin{multline}
\label{asjcas9}
\II{(\sum_{j=0}^i \overline{x}_{j,2}(t) + C\epsilon, \sum_{j=0}^i \overline{x}_{j,2}(t) - C\epsilon]}{D_n}
\le
\II{(Y_{i-1}^N(t_n^-),Y_i^N(t_n^-)]}{D_n} \\
\le
\II{(\sum_{j=0}^i \overline{x}_{j,2}(t) - C\epsilon,\sum_{j=0}^i \overline{x}_{j,2}(t) + C\epsilon]}{D_n}
\end{multline}
Let $\Gamma$ denote the LHS of the first equation in Lemma~\ref{lemma_ascasji1}.
Applying Lemma~\ref{GC_F}, we obtain
\begin{align*}
\Gamma
\le&  \lim_{\epsilon\downarrow 0}\lim_{k\to\infty}
\sum_{{\scriptstyle n=\mathcal{N}_\phi(N_k t)+1:}\atop{\scriptstyle W_n=1}}^{\mathcal{N}_\phi(N_k (t+\epsilon))}
 \frac{\II{(\sum_{j=0}^i \overline{x}_{j,2}(t) - C\epsilon,\sum_{j=0}^i \overline{x}_{j,2}(t) + C\epsilon]}{D_n}}{\epsilon N_k}
\\
=& \overline{x}_{i,2}(t)
\end{align*}
and using \eqref{asjcas9} in the other direction we obtain $\Gamma=\overline{x}_{i,2}(t)$ as desired.

\subsection{Proof of Lemma~\ref{lemma_ascasji2}}

We recall that we have analyzed $\overline{x}$ along a fixed $\omega \in \mathcal{C}$, where $\pp(\mathcal{C})=1$. We now explicit the dependence on $\omega$ and treat quantities $\overline{x}(t)$ and $X^N(t)$ as random variables.
Let
\begin{align}
Z_n^N \bydef \II{(Y_{i-1}^{N}(t_n^-),Y_{i}^{N}(t_n^-)]}{A_n^1 (1-X_{0,0}^{N}(t_n^-)-X_{0,1}^{N_k}(t_n^-) )}
\I{\sum_{j=0}^d X_{j,2}^{N}(t_n^-)>0}.
\end{align}
For all $n$, the random variable
$Z_n^N$
is  $\mathcal{F}_n$-measurable where
$\mathcal{F}_n \bydef \{X^{N}(t_n^{N,\lambda-}), A_n^1,W_n\}$, and
\begin{align}
\label{asjc7h6ftc}
\mathbb{E}[Z_{n}^N| \mathcal{F}_n\setminus A_n^1]
=
\frac{X_{i,2}^{N}(t_n^-)}{1 - X_{0,0}^{N}(t_n^-) - X_{0,1}^{N}(t_n^-)}
\I{\sum_{j=0}^d X_{j,2}^{N}(t_n^-)>0}
\end{align}
where the set $\mathcal{F}_n\setminus W_n$ denotes the set $\mathcal{F}_n$ with $A_n^1$ removed.
Now, let
$\Delta_n^N\bydef Z_n^N - \mathbb{E}[Z_n^N| \mathcal{F}_n\setminus W_n]$.
Then, $\mathbb{E}[\Delta_n^N|\mathcal{F}_n\setminus W_n]=0$ and $|\Delta_n^N|\le 2$, and applying the Azuma--Hoeffding inequality, we get
\begin{equation}
\mathbb{P} \left( \frac{1}{N} \left|\sum_{n=1}^N \Delta_n^N  \right| >\delta \right) \le 2 \exp\left( -\frac{(N\delta)^2}{ 8 N} \right)
\end{equation}
for any $\delta>0$.
Since $\sum_N \exp\left( -{N \delta^2/ 8} \right)<\infty$, an application of the Borel--Cantelli lemma shows that
$\frac{1}{N} \sum_{n=1}^N \Delta_n^N \to 0$ almost surely.
In particular,
\begin{multline}
\label{vjuygedffheryf6723ws}
\lim_{N\to\infty} \frac{1}{\epsilon N}
\sum_{{\scriptstyle n=\mathcal{N}_\phi(N t) +1:}\atop{\scriptstyle W_n=0}}^{\mathcal{N}_\phi(N (t+\epsilon)) }
\II{(Y_{i-1}^{N}(t_n^-),Y_{i}^{N}(t_n^-)]}{A_n^1 (1-X_{0,0}^{N}(t_n^-)-X_{0,1}^{N}(t_n^-) )}
\\
\times \left( \I{\sum_{j=0}^d X_{j,2}^{N}(t_n^-)>0}
-
\mathbb{E}[Z_n^N| \mathcal{F}_n\setminus W_n]
\right)
=0
\end{multline}
almost surely.
We now come back to work on a given trajectory $\omega$.
In view of the previous equality,
we may redefine $\mathcal{C}$ in Lemma~\ref{GC_F} to be a subset of
$\mathcal{C}'$ where $\pp(\mathcal{C}'=1)$
and \eqref{vjuygedffheryf6723ws} holds for all $\omega \in\mathcal{C}'$.
Therefore, we fix $\omega\in\mathcal{C}$ and use \eqref{eq:2_L_eps} and \eqref{asjc7h6ftc} to obtain that
\begin{align*}
& \lim_{k\to\infty} \frac{1}{\epsilon N_k}
\sum_{{\scriptstyle n=\mathcal{N}_\phi(N_k t) +1:}\atop{\scriptstyle W_n=0}}^{\mathcal{N}_\phi(N_k (t+\epsilon)) }
\mathbb{E}[Z_n^{N_k}| \mathcal{F}_n\setminus W_n]
\I{\sum_{j=0}^d X_{j,2}^{N_k}(t_n^-)>0}\\
& \le
\lim_{k\to\infty} \frac{1}{\epsilon N_k}
\sum_{{\scriptstyle n=\mathcal{N}_\phi(N_k t) +1:}\atop{\scriptstyle W_n=0}}^{\mathcal{N}_\phi(N_k (t+\epsilon)) }
\frac{\overline{x}_{i,2}(t)+\delta}{1-\overline{x}_{0,0}(t)-\overline{x}_{0,1}(t)-\delta} \\
& \qquad \times \I{\sum_{j=0}^d X_{j,2}^{N_k}(t_n^-)>0}
\end{align*}
for any $\delta>0$ sufficiently small.
Replacing $\delta$ by $-\delta$ in the last fraction term,
the previous inequality can be reversed and letting $\delta\downarrow 0$,
we obtain
\begin{footnotesize}
\begin{multline}
\label{asjc8as9}
\lim_{k\to\infty} \frac{1}{\epsilon N_k}
\sum_{{\scriptstyle n=\mathcal{N}_\phi(N_k t) +1:}\atop{\scriptstyle W_n=0}}^{\mathcal{N}_\phi(N_k (t+\epsilon)) }
\mathbb{E}[Z_n^{N_k}| \mathcal{F}_n\setminus W_n]
\I{\sum_{j=0}^d X_{j,2}^{N_k}(t_n^-)>0}\\
=
\frac{\overline{x}_{i,2}(t)}{1-\overline{x}_{0,0}(t)-\overline{x}_{0,1}(t)}
\lim_{k\to\infty}
\sum_{{\scriptstyle n=\mathcal{N}_\phi(N_k t) +1:}\atop{\scriptstyle W_n=0}}^{\mathcal{N}_\phi(N_k (t+\epsilon)) }
 \frac{\I{\sum_{j=0}^d X_{j,2}^{N_k}(t_n^-)>0}}{\epsilon N_k}
\end{multline}
\end{footnotesize}

Finally, \eqref{asjc8as9} and \eqref{vjuygedffheryf6723ws} give \eqref{asjc78dcckiujff}.

\subsection{Proof of Lemma~\ref{asjch7rrr}}
Let $w_d\bydef \sum_{j=0}^d x_{j,2}$.
For now, let us assume that $x_{0,2}>0$.
In this case, $w_d>0$ and \eqref{standard_FP} boils down to (by~(6))
\begin{subequations}
\label{standard_FP_2}
\begin{align}
x_{1,2} &= \tfrac{ \lambda}{w_d} x_{0,2}\\
x_{i+1,2} &= x_{i,2} + \tfrac{\lambda }{w_d}(x_{i,2}-x_{i-1,2}),\quad i=1,\ldots,d\\
x_{d+2,2} &= x_{d+1,2} - \tfrac{\lambda }{w_d} x_{d,2}\\
x_{i+1,2} &= x_{i,2},\qquad\qquad\qquad\qquad\quad\,\,\,  i\ge d+2.
\end{align}
\end{subequations}
Since $\|x\|=1$, \eqref{standard_FP_2} holds if and only if  $x_{i,2}=0$ for all $i\ge d+2$ and
\begin{align}
\label{ascscs}
x_{i,2} = \left( \tfrac{ \lambda}{w_d}\right)^i x_{0,2},\qquad   i = 0,\ldots,d+1.
%
%
%
\end{align}
If $d=0$, then $w_d=x_{0,2}$ and $x_{1,2}=\lambda$, and the lemma is proven.
Thus, let $d\ge 1$.
Summing \eqref{ascscs} over $i=0,\ldots,d$, we obtain
\begin{align}
w_d = \frac{1-\left( \frac{ \lambda}{w_d}\right)^{d+1}}{1-\frac{ \lambda}{w_d} }   x_{0,2}
%
\end{align}
and letting $z_d\bydef\sum_{j=1}^{d} x_{j,2}$ we obtain (10) as desired and it remains to prove (7a). Using \eqref{ascscs},
we notice that (10) holds if and only if
\begin{align}
z_d+x_{0,2} = \frac{x_{0,2}-x_{d+1,2}}{1-\frac{ \lambda}{z_d+x_{0,2}} }
\end{align}
and rearranging terms we obtain $z_d+x_{0,2}- \lambda = x_{0,2}-x_{d+1,2} $.
Then, (7a) follows by using the normalizing condition $\|x\|=1$ as $x_i=0$ for all $i\ge d+2$.

It remains to consider the case where $x_{0,2}=0$. Here, $x_{0,2}=0$ if and only if $x_{0,1}=0$, by \eqref{ajvdfu8bf}, which implies
\begin{align}
\label{sc8as9sd}
g\I{{x}_{0,0}>0}=0,
\end{align}
by \eqref{x01_FP}.
In addition, if $w_d>0$, then \eqref{standard_FP} boils down again to \eqref{standard_FP_2} and $x_{0,2}=0$ would imply that $x_{i,2}=0$ for all $i$. This is not possible in view of $\|x\|=1$ and, therefore, we must have $w_d=0$.
Since necessarily $x_{0,0}<1$,
(6) simplifies to
\begin{align}
h_i (x) =
\left\{
\begin{array}{ll}
%
{x}_{d+1,2} \I{i=d} \I{{x}_{d+1,2} \le \lambda}
& {\rm if }\,\, i\le d,\, \\
 \frac{{x}_{i,2}}{1-{x}_{0,0}} (\lambda - {x}_{d+1,2}  )^+
& {\rm if }\,\, i> d
\end{array}
\right.
\end{align}
and substituting in
\eqref{sj8x} we get
\begin{subequations}
\label{standard_FP_2bb}
\begin{align}
x_{i,2} &= 0, \quad i\le d\\
\label{askcygg}
x_{d+1,2} &= {x}_{d+1,2} \I{{x}_{d+1,2} \le \lambda}\\
%
%
\label{askcyggC}
x_{i,2} & = \frac{{x}_{i-1,2}}{1-{x}_{0,0}} (\lambda - {x}_{d+1,2}  )^+, \quad i\ge d+2.
\end{align}
\end{subequations}
This gives~(9a).
Now, if ${x}_{d+1,2} > \lambda$, then \eqref{askcygg} is violated,
and if ${x}_{d+1,2}=0$, then \eqref{askcyggC} and $\|x\|=1$ give the contradiction that $1=\lambda$.
So, necessarily ${x}_{d+1,2} \in(0, \lambda]$, i.e., (9c).
Here, we notice that ${x}_{d+1,2}$ is not tied to a specific value.
Then, summing \eqref{askcyggC} we obtain
\begin{align}
\sum_{i\ge d+2} x_{i,2} & = \sum_{i\ge d+2}\frac{{x}_{i-1,2}}{1-{x}_{0,0}} (\lambda - {x}_{d+1,2}  ),
\end{align}
which, using $\|x\|=1$ and \eqref{standard_FP_2bb}, holds if and only if
\begin{align}
1 -x_{d+1,2}-x_{0,0} & = \frac{\lambda - {x}_{d+1,2} }{1-{x}_{0,0}} (1 -x_{0,0})
\end{align}
i.e., if and only if $x_{0,0}=1-\lambda$;
note that $x_{0,0}=0$ is not possible as otherwise  \eqref{askcyggC} and $\|x\|=1$ give the contradiction that $1<\lambda$.
Since $x_{0,0}>0$, necessarily $g=0$ by \eqref{sc8as9sd}, which gives~(9d).
Using ${x}_{d+1,2} \le \lambda$ and $x_{0,0}=1-\lambda$ in~\eqref{askcyggC}, we obtain
$x_{d+2,2}  = x_{d+1,2}  \left(1 - {{x}_{d+1,2}/\lambda}  \right) $
%
%
and applying inductively~\eqref{askcyggC}, we obtain~(9b).
This concludes the proof.

\section{Proof of Proposition~1}
\label{sec:proof_prop}

The fact that $x^\star$ is a fixed point is trivial.
Suppose that there exists $\delta>0$ such that $x_{0,2}(t)>0$ on $(0,\delta]$.
Then, there exists $\delta'>0$ such that $\dot{x}_{0,2}(t)>0$ on $(0,\delta']$.
Using~(4c), which gives
$\dot{x}_{0,2}
= x_{1,2} - \lambda + \beta x_{0,1} - \gamma x_{0,2},
$
we obtain
\begin{align}
x_{1,2}(t) + \beta x_{0,1}(t) > \lambda + \gamma x_{0,2}(t),\quad \forall t\in(0,\delta']
\end{align}
and thus $x_{1,2}(0) + \beta x_{0,1}(0) = \lim_{t\downarrow 0 } x_{1,2}(t) + \beta x_{0,1}(t) \ge \lambda$, by continuity of the fluid model.
This contradicts the last condition in (15) and thus $x_{0,2}(t)=0$ on a right neighborhood of zero, say $[0,\delta]$.
Since $g(x)
=\lambda - 1+x_{0,0}
$,
on $[0,\delta]$ we obtain (using~(4))
\begin{subequations}
\label{gx_proof_prop1}
\begin{align}
\label{gx_1_prop1}
\dot{{x}}_{0,0} & = - \alpha (\lambda - 1+x_{0,0})\\
\label{gx_2_prop1}
\dot{{x}}_{0,1} & = \alpha  (\lambda - 1+x_{0,0})  -\beta{x}_{0,1}\\
\label{gx_3_prop1}
\dot{x}_{0,2} &= 0 \\
\label{gx_4_prop1}
\dot{x}_{i,2} &= x_{i+1,2}-x_{i,2} + h_{i-1} (x)- h_i(x),\quad i\ge 1
\end{align}
\end{subequations}
where
\begin{align}
h_i (x) =
\left\{
\begin{array}{ll}
\beta {x}_{0,1}+{x}_{1,2}
& {\rm if } \,\, i=0,\\
 \frac{{x}_{i,2}}{y_1} (\lambda - {x}_{1,2} - \beta {x}_{0,1}  )^+
& {\rm if } \,\, i>0.\\
%
%
%
\end{array}
\right.
\end{align}
We observe that~\eqref{gx_1_prop1}-\eqref{gx_2_prop1} form an autonomous linear ODE system.
By continuity of $x(t)$, (15) holds as well on a right neighborhood of zero.
Now, we actually show that (15) holds on $[0,\infty)$, i.e., $\delta=+\infty$.
Towards this purpose,
let us analyze the system~\eqref{gx_1_prop1}-\eqref{gx_2_prop1} in isolation.
After some algebra, we obtain
\begin{subequations}
\label{asc90as}
\begin{align}
%
x_{0,0}(t) & = 1-\lambda + (x_{0,0} -1 + \lambda) e^{-\alpha t}\\
x_{0,1}(t) & =
\frac{\alpha (x_{0,0} -1 + \lambda)}{\beta-\alpha}( e^{-\alpha t} - e^{-\beta t} )
+ x_{0,1}(0) e^{-\beta t} .
\end{align}
\end{subequations}
Thus,
\begin{itemize}
\item[i)] $x_{0,0}(t)$ monotonically decreases to zero  as $t\to\infty$, and

\item[ii)] $y_0(t)=y_1(t)<1$ with both $y_0(t)$ and $y_1(t)$ monotonically increasing to $\lambda$ because $\dot{x}_{0,0} +\dot{x}_{0,1}$ is always non-increasing and $x_{0,2}$ stays on zero.
\end{itemize}
To prove that (15) holds on $[0,\infty)$, it remains to show that
$x_{1,2}(t)+\beta x_{0,1}(t)< \lambda$ for all $t\ge 0$.
This property is true because
$
x_{1,2} + \beta x_{0,1}
\le y_{1} +  x_{0,1}
= 1-x_{0,0}
= \lambda - (x_{0,0}(0)-1+\lambda) e^{-\alpha t}< \lambda
%
$.
Thus, $x(t)$ satisfies \eqref{gx_proof_prop1}
on $[0,\infty)$.
In addition, since $x_{0,0}(0)+x_{0,1}(0)<1$ and $\dot{x}_{0,0}(t)+\dot{x}_{0,1}(t)=-\beta x_{0,1}(t)\le 0$ for all $t$, the drift function of \eqref{gx_proof_prop1} is Lipschitz and therefore it induces a unique flow [11,page~56].
Since $x_{0,0}(t)\downarrow 1-\lambda$ as $t\to\infty$, for all $t\ge 0$
\begin{multline}
\label{ascascssss}
\dot{\overline{Q}}(x(t)) =   \lambda - y_1(t) =   x_{0,0}(t)+x_{0,1}(t) + \lambda - 1
\\
\ge   x_{0,0}(t)  + \lambda - 1 >0,
\end{multline}
where the first equality follows by Lemma~\ref{lemma:Qdot}.
In particular, $\lim_{t\to\infty} Q(x(t))$ exists and must be greater than $\lambda$ because $\lambda<\overline{Q}(x(0))<\infty$.
Combining \eqref{asc90as} and~\eqref{ascascssss}, we obtain
\begin{align*}
 \dot{\overline{Q}}(x(t))
& =
 (x_{0,0}(0)-1+\lambda) e^{-\alpha t}
 + x_{0,1}(0) e^{-\beta t}  \\
 &\quad + \frac{\alpha (x_{0,0}(0)-1+\lambda)}{\beta-\alpha}( e^{-\alpha t} - e^{-\beta t} )\\
%
& =
\underbrace{\frac{\beta(x_{0,0}(0)-1+\lambda)}{\beta-\alpha}}_{\bydef C_1} e^{-\alpha t} \\
&\quad
+ \underbrace{\left(x_{0,1}(0) - \frac{\alpha (x_{0,0}(0)-1+\lambda)}{\beta-\alpha}\right)}_{\bydef C_2} e^{-\beta t} .
\end{align*}
Integrating,
\begin{align*}
\overline{Q}(x(t)) =  & \, \overline{Q}(x(0))
+  \frac{C_1}{\alpha} (1-e^{-\alpha t})
+  \frac{C_2}{\beta} (1-e^{-\beta t}) \\
 \xrightarrow[t\to\infty]{} & \,
\overline{Q}(x(0))
+  \frac{\alpha+\beta}{\alpha\beta} (x_{0,0}(0)-1+\lambda)
+  \frac{1}{\beta}x_{0,1}(0)
%
\end{align*}
which proves~(17).
Finally, suppose that $\lim_{t\to\infty} x_{1,2}(t)$ exists, say $x_{1,2}(\infty)$.
Then, necessarily $x_{1,2}(\infty)<\lambda$ because
$y_1(t)\to\lambda$ and
$\lim_{t\to\infty} Q(x(t))>\lambda$ excludes that $x_{1,2}(t)\to\lambda$. Then, using \eqref{gx_4_prop1} when $i=1$ and that $x_{1,2}(t)$ is Lipschitz continuous,
\begin{align*}
0 & =\lim_{t\to\infty} \dot{x}_{1,2}(t) \\
&= \lim_{t\to\infty} x_{2,2} + \beta x_{0,1} - \frac{{x}_{1,2}}{1-{x}_{0,0}-{x}_{0,1}} (\lambda - {x}_{1,2} - \beta {x}_{0,1}  )\\
%
& = \lim_{t\to\infty} \left( x_{2,2} - \frac{{x}_{1,2}}{\lambda} (\lambda - {x}_{1,2}  ) \right) \\
& = -x_{1,2}(\infty)  \left(1 - \frac{{x}_{1,2}(\infty)}{\lambda} \right ) + \lim_{t\to\infty} x_{2,2},
\end{align*}
which shows that $\lim_{t\to\infty} x_{2,2}$ must exists as well and be equal to $x_{1,2}(\infty)  \left(1 - \frac{{x}_{1,2}(\infty)}{\lambda} \right )$.
By induction, $\lim_{t\to\infty} x_{i,2}$ exists and is equal to $x_{i,2}(\infty)  \left(1 - \frac{{x}_{1,2}(\infty)}{\lambda} \right )^{i-1}$.
Thus, $x(\infty)\in \SO$.


%

\review{
\section{Additional material supporting numerical simulations}

Table~\ref{tab:my-table} reports the numerical values of $\mathcal{R}_{{\rm Wait}}$ and $\mathcal{R}_{{\rm Energy}}$ plotted in Figure 2.
\begin{table}[h]
\begin{center}
 \begin{tabular}{lccc}
\multicolumn{4}{c}{$\lambda=0.35$}                                                                                                                  \\
\multicolumn{1}{c}{} & $d=1$                 & $d=5$                 & $d=10$                \\ \cline{2-4}
\multicolumn{1}{l|}{$N=100$}                                                & \multicolumn{1}{c|}{0.01786,\,1.00296} & \multicolumn{1}{c|}{0.00140,\,1.03773} & \multicolumn{1}{c|}{0.001046,\,1.00762} \\ \cline{2-4}
\multicolumn{1}{l|}{$N=500$}                                                & \multicolumn{1}{c|}{0.00674,\,1.00271} & \multicolumn{1}{c|}{0.00031,\,1.01113} & \multicolumn{1}{c|}{0.000013,\,1.00956} \\ \cline{2-4}
\multicolumn{1}{l|}{$N=1000$}                                               & \multicolumn{1}{c|}{0.00400,\,1.00387} & \multicolumn{1}{c|}{0.00022,\,1.00492} & \multicolumn{1}{c|}{0.000007,\,1.00554} \\ \cline{2-4}
\\
\hline
\\
\multicolumn{4}{c}{$\lambda=0.7$}                                                                                                                  \\
\multicolumn{1}{c}{} & $d=1$                 & $d=5$                 & $d=10$                \\ \cline{2-4}
\multicolumn{1}{l|}{$N=100$}                                                & \multicolumn{1}{c|}{0.01414,\,1.00091} & \multicolumn{1}{c|}{0.01086,\,1.00127} & \multicolumn{1}{c|}{0.010230,\,1.00284} \\ \cline{2-4}
\multicolumn{1}{l|}{$N=500$}                                                & \multicolumn{1}{c|}{0.00250,\,1.00200} & \multicolumn{1}{c|}{0.00024,\,1.00081} & \multicolumn{1}{c|}{0.000158,\,1.00153} \\ \cline{2-4}
\multicolumn{1}{l|}{$N=1000$}                                               & \multicolumn{1}{c|}{0.00162,\,1.00234} & \multicolumn{1}{c|}{0.00011,\,1.00355} & \multicolumn{1}{c|}{0.000025,\,1.00285} \\ \cline{2-4}
\end{tabular}
\end{center}
\caption{Numerical values of $(\mathcal{R}_{{\rm Wait}},\mathcal{R}_{{\rm Energy}})$ in Figure 2.}
\label{tab:my-table}
\end{table}
}

\end{document}